\documentclass[11pt,oneside]{article}
\usepackage{amssymb,amsmath,latexsym,amsfonts,amscd,amsthm,multirow,ctable,colortbl,mathdots}
\usepackage{fancyhdr}
\usepackage[headings]{fullpage}
\usepackage{stmaryrd}
\usepackage[all]{xy}
\usepackage{rotating}

\usepackage{setspace}

\usepackage{hyperref}

\def\c{\gamma}

\def\t{\tau}

\newcommand{\gt}[1]{\mathfrak{#1}}

\newcommand{\jmod}[1]{\hspace{-1.5ex}\pmod{#1}}

\newcommand{\RR}{{\mathbb R}}
\newcommand{\CC}{{\mathbb C}}
\newcommand{\ZZ}{{\mathbb Z}}
\newcommand{\QQ}{{\mathbb Q}}
\newcommand{\HH}{{\mathbb H}}

\newcommand{\ii}{{\bf i}}

\newcommand{\lab}{{\langle}}    
\newcommand{\rab}{{\rangle}}    

\newcommand{\End}{\operatorname{End}}
\newcommand{\Aut}{\operatorname{Aut}}

\newcommand{\Id}{\operatorname{Id}}
\newcommand{\tr}{\operatorname{{tr}}}

\newcommand{\sgn}{\operatorname{sgn}}

\newcommand{\sdim}{\operatorname{sdim}}

\newcommand{\Cliff}{\operatorname{Cliff}}

\newcommand{\SL}{\operatorname{\textsl{SL}}}      
\newcommand{\GL}{{\textsl{GL}}}      

\newcommand{\vn}{V^{\natural}} 
\newcommand{\G}{\Gamma}	

\newcommand{\vac}{{\bf v}}
\newcommand{\tw}{{\rm tw}}

\renewcommand{\le}{\varepsilon}

\newtheorem{thm}{Theorem}[section]
\newtheorem{cor}[thm]{Corollary}

\newtheorem{prop}[thm]{Proposition}

\theoremstyle{definition}

\theoremstyle{remark}
\newtheorem{rmk}[thm]{Remark}

\numberwithin{equation}{section}

\begin{document}

\setstretch{1.4}

\title{
    \textsc{\huge{
    {T}he {U}mbral {M}oonshine {M}odule\\ for the {U}nique {U}nimodular\\ {N}iemeier {R}oot {S}ystem
    }}
    }

\renewcommand{\thefootnote}{\fnsymbol{footnote}} 
\footnotetext{\emph{MSC2010:} 11F22, 11F37, 17B69, 20C35.}     
\renewcommand{\thefootnote}{\arabic{footnote}} 


\author{
	John F. R. Duncan\footnote{
         Department of Mathematics, Applied Mathematics and Statistics,
         Case Western Reserve University,
         Cleveland, OH 44106,
         U.S.A.
         \newline\indent\indent
         {\em E-mail:} {\tt john.duncan@case.edu}
         }\\
	Jeffrey A. Harvey\footnote{
	Enrico Fermi Institute and Department of Physics,
         University of Chicago,
         Chicago, IL 60637,
         U.S.A.
         \newline\indent\indent
         {\em E-mail:} {\tt j-harvey@uchicago.edu}
	}
}

\date{2014 December 28}

\maketitle

\abstract{
We use canonically-twisted modules for a certain super vertex operator algebra to construct the umbral moonshine module for the unique Niemeier lattice that coincides with its root sublattice. In particular, we give explicit expressions for the vector-valued mock modular forms attached to automorphisms of this lattice by umbral moonshine. We also characterize the vector-valued mock modular forms arising, in which four of Ramanujan's fifth order mock theta functions appear as components.
}

\clearpage

\section{Introduction}

In his 2002 Ph.D thesis \cite{Zwegers} Zwegers gave an intrinsic definition of mock theta functions and provided new insight into three
families of such 
functions, constructed 
\begin{enumerate}
\item 
in terms of Appell--Lerch sums, 
\item 
as the Fourier coefficients of meromorphic Jacobi forms,
and 
\item 
via theta functions attached to cones in lattices of indefinite signature. 
\end{enumerate}
The first two constructions have played a central role in 
recently observed moonshine connections between finite groups and mock theta functions. These started with the observation in \cite{Eguchi2010} that the
elliptic genus of a K3 surface has a decomposition into characters of the $N=4$ superconformal algebra with multiplicities that at low levels
are equal to the dimensions of irreducible representations of the Mathieu group $M_{24}$.  Appell--Lerch sums appear in this analysis in the
so called ``massless" characters. This Mathieu moonshine connection was conjectured in \cite{UM,MUM} to be part of a much more general phenomenon,
known as umbral moonshine, which attaches a vector-valued mock modular form $H^X$, a finite group $G^X$, and an infinite-dimensional graded $G^X$-module $K^X$ to the root systems of each of the $23$ Niemeier lattices. The analysis in \cite{MUM} relied heavily on the construction of
mock modular forms in terms of meromorphic Jacobi forms and built on the important work in \cite{Dabholkar:2012nd} extending the analysis of \cite{Zwegers}
and characterizing special Jacobi forms in terms of growth conditions.

Whilst the existence of the $G^X$-modules $K^X$ has now been proved \cite{Gannon:2012ck,umrec} for all Niemeier root systems $X$, no explicit construction of the modules $K^X$ is yet known. 

However, in this paper
we construct the $G^X$-module $K^X$ for the case that $X=E_8^3$. 
To do so
we employ the third characterization of mock theta functions in terms of indefinite theta functions. This 
enables
us to 
employ the formalism of vertex operator algebras \cite{Bor_PNAS,FLM} which has been so fruitfully employed (in \cite{FLM,borcherds_monstrous} to name just two) in the understanding of monstrous moonshine \cite{MR554399,Tho_FinGpsModFns,Tho_NmrlgyMonsEllModFn}. 

See \cite{mnstmlts} for a recent review of moonshine both monstrous and umbral, and many more references on these subjects.

To explain the methods of this paper in more detail, we first recall the {\em Pochammer symbol}
\begin{gather}\label{eqn:intro-poch}
(x;q)_n :=\prod_{k=0}^{n-1} (1-xq^k), 
\end{gather}
and the fifth order mock theta functions
\begin{gather}
\begin{split}
\chi_0(q)&:=\sum_{n\geq 0} \frac{q^{n} }{(q^{n+1};q)_n},\\
\chi_1(q)&:= \sum_{n\geq 0} \frac{q^{n} }{(q^{n+1};q)_{n+1}},
\end{split}
\end{gather}
from Ramanujan's last letter to Hardy \cite{MR2280843,MR947735}. 
The conjectures of \cite{MUM} (see also \cite{mumcor}) imply 
the existence of a bi-graded super vector space $K^X=\bigoplus_r K^X_r=\bigoplus_{r,d}K^{X}_{r,d}$ that is a module for $G^X\simeq S_3$ and satisfies
\begin{gather}\label{eqn:intro:KXmocktheta}
\begin{split}
	\sdim_qK^X_1&=-2q^{-1/120}+\sum_{n>0}\dim K^X_{1,n-1/120}q^{n-1/120}=2q^{-1/120}(\chi_0(q)-2),\\
	\sdim_qK^X_7&=\sum_{n>0}\dim K^X_{7,n-49/120}q^{n-49/120}=2q^{71/120}\chi_1(q).
\end{split}
\end{gather}
Here $\sdim_q V:=\sum_{n}(\dim (V_{\bar 0})_n-\dim (V_{\bar 1})_n)q^n$ for $V$ a $\QQ$-graded super space with even part $V_{\bar 0}$ and odd part $V_{\bar 1}$.

According to work \cite{MR2558702} of Zwegers, we have identities
\begin{gather}\label{eqn:intro:zwegers}
\begin{split}
	2-\chi_0(q)&=\frac{1}{(q;q)_{\infty}^2}\left(\sum_{k,l,m\geq 0}+\sum_{k,l,m<0}\right)(-1)^{k+l+m}q^{(k^2+l^2+m^2)/2+2(kl+lm+mk)+(k+l+m)/2},\\
	\chi_1(q)&=\frac{1}{(q;q)_{\infty}^2}\left(\sum_{k,l,m\geq 0}+\sum_{k,l,m<0}\right)(-1)^{k+l+m}q^{(k^2+l^2+m^2)/2+2(kl+lm+mk)+3(k+l+m)/2},
\end{split}
\end{gather}
where $(x;q)_{\infty}:=\prod_{n\geq 0}(1-xq^n)$. In this article we use (\ref{eqn:intro:zwegers}) as a starting point for the construction of a super vertex operator algebra $V^X$ (cf. (\ref{eqn:va:cnstn-VX})). We show that canonically-twisted modules for $V^X$, constructed explicitly in \S\ref{sec:va:cnstn} (cf. (\ref{eqn:va:cnstn-Vtw})), furnish a bi-graded $G^X$-module for which the graded trace functions are exactly compatible with the predictions of \cite{MUM}. In other words, we construct the analogue of the moonshine module $\vn$, of Frenkel--Lepowsky--Meurman \cite{FLMBerk,FLMPNAS,FLM}, for the $X=E_8^3$ case of umbral moonshine. 

To prove that our construction is indeed the $X=E_8^3$ counterpart to $\vn$, we verify the $X=E_8^3$ analogue of the Conway--Norton moonshine conjecture, proven by Borcherds in \cite{borcherds_monstrous} in the case of the monster, which predicts that the trace functions arising are uniquely determined by their automorphy and their asymptotic behavior near cusps. Thus we verify the $X=E_8^3$ analogues of both of the two major conjectures of monstrous moonshine. 

To prepare for precise statements of results, recall that vector-valued functions $H^X_g(\tau)=(H^X_{g,r}(\tau))$ on the upper half plane $\HH$ are considered in \cite{MUM}, for $g\in G^X\simeq S_3$, where the components are indexed by $r\in \ZZ/60\ZZ$. 
Define $o(g)$ to be the order of an element $g\in G^X$. 
The $H^X_g$ are not uniquely determined in \cite{MUM}, except for the case that $g=e$ is the identity, $o(g)=1$. But it is predicted that $H^X_g$ is a mock modular form of weight $1/2$ for $\Gamma_0(o(g))$, with shadow given by a certain vector-valued unary theta function $S^X_g$ (cf. (\ref{eqn:mcktht-SXg})), and specified polar parts at the cusps of $\Gamma_0(o(g))$. In more detail, $H^X_g$ should have the same polar parts as $H^X:=H^X_e$ at the infinite cusp of $\Gamma_0(o(g))$, but should have vanishing polar parts at any non-infinite cusps. In practice, this amounts to the statement that we should have
\begin{gather}
	H^X_{g,r}(\tau)=\begin{cases} \mp 2q^{-1/120}+O(q^{119/120}),&\text{ if $r=\pm 1,\pm11,\pm19,\pm29\pmod{60}$,}\\
	O(1),&\text{ otherwise,}
	\end{cases}
\end{gather}
for $q=e^{2\pi \ii\tau}$, and all components of $H^X_g(\tau)$ should remain bounded as $\tau\to 0$, if $g\neq e$. (For if $g\neq e$ then $o(g)=2$ or $o(g)=3$, and then $\Gamma_0(o(g))$ has only one cusp other than the infinite one, and this is the cusp represented by $0$.)

Our main result is the following, where the functions $T^X_g$ are defined in \S\ref{sec:mcktht:um} (cf. (\ref{eqn:mcktht-TXg})) in terms of traces of operators on canonically-twisted modules for $V^X$. 
\begin{thm}\label{thm:intro-maintheorem}
Let $g\in G^X$. 
If $o(g)\neq 3$ then  $2T^X_{g}$ is the Fourier expansion of the unique vector-valued mock modular form of weight $1/2$ for $\Gamma_0(o(g))$ whose shadow is $S^X_g$, and whose polar parts coincide with those of $H^X_g$. If $o(g)=3$ then $2T^X_g$ is the Fourier expansion of the unique vector-valued modular form of weight $1/2$ for $\Gamma_0(3)$ which has the multiplier system $\rho_{3|3}\overline{\sigma^X}$, and polar parts coinciding with those of $H^X_g$. 
\end{thm}
Here $\sigma^X:\SL_2(\ZZ)\to\GL_{60}(\CC)$ denotes the multiplier system of $S^X:=S^X_e$ (cf. (\ref{eqn:mcktht:um-sigmaX})), and $\rho_{3|3}:\Gamma_0(3)\to \CC^\times$ is defined in (\ref{eqn:mcktht:um-rho33}).

Armed with Theorem \ref{thm:intro-maintheorem}, we may now define the $H^X_g$ concretely and explicitly, for $g\in G^X$, by setting
\begin{gather}\label{eqn:intro-HXg}
	H^X_g(\tau):=2T^X_g(\tau),
\end{gather}
where $T^X_g(\tau)$ denotes the function obtained by substituting $e^{2\pi \ii \tau}$ for $q$ in the series expression (\ref{eqn:mcktht-TXg}) for $T^X_g$.

Expressions for the components of $H^X_g$ are given in \S5.4 of \cite{MUM}, in terms of fifth order mock theta functions of Ramanujan, for the cases that $o(g)=1$ and $o(g)=2$, but it is not verified there that these prescriptions define mock modular forms with the specified shadows. Our work confirms these statements, as the following theorem demonstrates. 

\begin{thm}\label{thm:intro-rammcktht}
We have the following identities.
\begin{gather}
H^{X}_{1A,1}(\tau) 
=\begin{cases}\label{eqn:intro-rammcktht1}
	\pm 2q^{-1/120} \left( \chi_0(q) - 2 \right),&\text{ if $r=\pm 1,\pm 11,\pm 19,\pm 29$,} \\ 
	\pm 2q^{71/120} \chi_1(q), &\text{ if $r=\pm 7,\pm 13,\pm 17,\pm 27$.}  
	\end{cases}\\
H^{X}_{2A,1}(\tau) 
=\begin{cases}\label{eqn:intro-rammcktht2}
	\mp 2 q^{-1/120}  \phi_0(-q), &\text{ if $r=\pm 1,\pm 11,\pm 19,\pm 29$,} \\ 
	\pm 2 q^{-49/120} \phi_1(-q), &\text{ if $r=\pm 7,\pm 13,\pm 17,\pm 27$.}  
	\end{cases}
\end{gather}
\end{thm}
The fifth order mock theta functions $\phi_0$ and $\phi_1$ were defined by Ramanujan (also in his last letter to Hardy), by setting
\begin{gather}
	\begin{split}\label{eqn:intro-phi01}
		\phi_0(q)&:=\sum_{n\geq 0} q^{n^2} {(-q;q^2)_n},\\
	\phi_1(q)&:=\sum_{n\geq 0} q^{(n+1)^2} {(-q;q^2)_n}.
	\end{split}
\end{gather}

The identities (\ref{eqn:intro-rammcktht1}) follow immediately from Theorem \ref{thm:intro-maintheorem}, since the the $V^X$-modules used to define the $T^X_g$ have been constructed specifically so as to make Zwegers' identity (\ref{eqn:intro:zwegers}) manifest. By contrast, the $o(g)=2$ case of Theorem \ref{thm:intro-rammcktht} requires some work, since the expressions we obtain naturally from our construction of $T^X_g$ do not obviously coincide with (\ref{eqn:intro-rammcktht2}). Thus the proof of Theorem \ref{thm:intro-rammcktht} entails non-trivial  $q$-series identities which may be of independent interest. 
\begin{cor}\label{cor:intro-qseriesid}
We have
\begin{gather}
\begin{split}\label{eqn:intro-qseriesid1}
&\left( \sum_{k,m \ge 0} -
\sum_{k,m <0} \right)_{k=m \pmod{2}}
(-1)^m q^{k^2/2+m^2/2+4km+k/2+3m/2} \\
&\qquad= {\prod_{n>0} (1+q^n)} \left( \sum_{k,m \ge 0} -
\sum_{k,m <0} \right) (-1)^{k+m}
q^{3 k^2+m^2/2 +4km+k+m/2},
\end{split}\\
\begin{split}\label{eqn:intro-qseriesid7}
&\left( \sum_{k,m \ge 0} -
\sum_{k,m <0} \right)_{k=m \pmod{2}}
(-1)^m q^{k^2/2+m^2/2+4km+3k/2+5m/2} \\
&\qquad= {\prod_{n>0} (1+q^n)} \left( \sum_{k,m \ge 0} -
\sum_{k,m <0} \right) (-1)^{k+m}
q^{3 k^2+m^2/2 +4km+3k+3m/2}.
\end{split}
\end{gather}
\end{cor}

The reader who is familiar with modularity results on trace functions attached to vertex operator algebras (cf. \cite{Zhu_ModInv,Dong2000,MR2046807}) and super vertex operator algebras (cf. \cite{DonZha_MdlrOrbVOSA}) may find it surprising that the functions we construct are (generally) mock modular, rather than modular, and have weight $1/2$, rather than weight $0$. In light of Zwegers' work \cite{Zwegers,MR2558702}, it is clear that we can obtain trace functions with mock modular behavior by considering vertex algebras constructed according to the usual lattice vertex algebra construction, but with a cone (or union of cones, cf. \S\ref{sec:va:cva}) 
taking on the role usually played by a lattice. A suitably chosen cone is the main ingredient for our construction of $V^X$. A general procedure for constructing super vertex operator algebras from cones in arbitrary signature is formalized in Theorem \ref{thm:va:cva-VD}. 

Note however that the cone vertex algebra construction does not, on its own, naturally give rise to trace functions with weight $1/2$. For this we introduce a single ``free fermion'' to the cone vertex algebra that we use to construct $V^X$, and we insert the zero mode (i.e. $L(0)$-degree preserving component) of the canonically-twisted vertex operator attached to a generator when we compute graded traces on canonically-twisted modules for $V^X$. In practice, this has the effect of multiplying the cone vertex algebra trace functions by $\eta(\tau):=q^{1/24}\prod_{n>0}(1-q^n)$. 

We remark that this technique may be profitably applied to other situations. For example, it is known (cf. e.g. \cite{MR1650637}) that the moonshine module $V^\natural$, when regarded as a module for the Virasoro algebra, is a direct sum of modules $L(h,24)$, for $h$ ranging over non-negative integers, satisfying
\begin{gather}
	\tr_{L(h,24)}q^{L(0)-c/24}=\begin{cases}
		(1-q)q^{-23/24}\eta(\tau)^{-1}&\text{ for $h=0$,}\\
		q^{h-23/24}\eta(\tau)^{-1}&\text{ for $h>0$,}
		\end{cases}	
\end{gather}
where $c=24$.
Also, the multiplicity of $L(0,24)$ is $1$, and the multiplicity of $L(1,24)$ is $0$.
Consequently, the weight $1/2$ modular form $\eta(\tau)J(\tau)$, with $J(\tau)=q^{-1}+O(q)$ the 
(so normalized) elliptic modular invariant, is almost the generating function of the dimensions of the homogeneous spaces of Virasoro highest weight vectors in $V^\natural$. Indeed, the actual generating function is just $q^{1/24}\eta(\tau)J(\tau)+1$. 

Certainly $\eta(\tau)J(\tau)$ has nicer modular properties than the Virasoro highest weight generating function of $V^\natural$, and moreover, an even more striking connection to the monster, as four of the dimensions of non-trivial irreducible representations for the monster appear as coefficients:
\begin{gather}\label{eqn:intro-fourans}
\eta(\tau)J(\tau)=\cdots+196883q^{25/24}+21296876q^{49/24}+842609326q^{73/24}+19360062527q^{97/24}+\cdots
\end{gather}
(cf. p.220 of \cite{atlas}).
This function $\eta(\tau)J(\tau)$ can be obtained naturally as a trace function on a canonically-twisted module for a super vertex operator algebra. For if we take $V$ to be the tensor product of $V^\natural$ with the super vertex operator algebra obtained by applying the Clifford module construction to a one dimension vector space (see \S\ref{sec:va:cliffmod} for details), then, choosing an irreducible canonically-twisted module $V_\tw$ for $V$, and denoting by $p(0)$ the coefficient of $z^{-1}$ in the canonically-twisted vertex operator attached to a suitably scaled element $p\in V$ with $L(0)p=\frac12 p$, we have
\begin{gather}\label{eqn:intro-VetaJ}
	\tr_{V_\tw}p(0)q^{L(0)-c/24}=\eta(\tau)J(\tau),
\end{gather}
where now $c=49/2$. (See \S\ref{sec:va:cliffmod} for more detail.) 

The importance of trace functions such as (\ref{eqn:intro-VetaJ}) within the broader context of modularity for super vertex operator algebras is analyzed in detail in \cite{MR3077918}. (See also \cite{MR3205090}.)

The organization of the paper is as follows. In \S\ref{sec:va} we recall some familiar constructions from vertex algebra and use these to construct the super vertex operator algebra $V^X$, and its canonically-twisted modules $V^{X,\pm}_{\tw,a}$, which play the commanding role in this work. We recall the lattice construction of super vertex algebras in \S\ref{sec:va:latva}, modules for lattice super vertex algebras in \S\ref{sec:va:latvamod}, and the Clifford module super vertex algebra construction in \S\ref{sec:va:cliffmod}. New material appears in \S\ref{sec:va:cva}, where we attach a super vertex operator algebra to a cone in an indefinite lattice. Using this, we formulate the construction of $V^X$ and the $V^{X,\pm}_{\tw,a}$ in \S\ref{sec:va:cnstn}. We also equip these spaces with $G^X$-module structure in \S\ref{sec:va:cnstn}, and compute explicit expressions (cf. Proposition \ref{prop:va:cnstn-tracefnexpressions}) for the graded traces of elements of $G^X$. 

In \S\ref{sec:mcktht} our focus moves from representation theory to number theory, as we seek to determine the properties of the graded traces arising from the action of $G^X$ on the $V^\pm_{\tw,a}$. We recall the relationship between mock modular forms and harmonic Maass forms in \S\ref{sec:mcktht:maass}, and we recall some results on Zwegers' indefinite theta series in \S\ref{sec:mcktht:indtht}. The proofs of our main results, Theorems \ref{thm:intro-maintheorem} and \ref{thm:intro-rammcktht}, are the content of \S\ref{sec:mcktht:um}.

We give tables with the first few coefficients of the $H^X_g$ in \S\ref{sec:coeffs}.

We frequently employ the notational convention $e(x):=e^{2\pi i x}$. 

\section{Vertex Algebra}\label{sec:va}

This section begins with a review of the lattice (super) vertex algebra construction in \S\ref{sec:va:latva}, and the natural generalization of this which defines lattice vertex algebra modules in \S\ref{sec:va:latvamod}. We review the special case of the Clifford module super vertex algebra construction we require in \S\ref{sec:va:cliffmod}. We introduce cone vertex algebras in \S\ref{sec:va:cva}, and put all of the preceding material together for the construction of $V^X$, and its canonically-twisted modules, in \S\ref{sec:va:cnstn}.

\subsection{Lattice Vertex Algebra}\label{sec:va:latva}

We briefly recall, following \cite{Bor_PNAS,FLM}, the standard construction which associates a super vertex algebra $V_L$ to a central extension of an integral lattice $L$. We also employ \cite{MR2082709} as a reference. Set $\gt{h}:=L\otimes_{\ZZ}\CC$, and extend the bilinear form on $L$ to a symmetric $\CC$-bilinear form on $\gt{h}$ in the natural way. Set $\hat{\gt{h}}:=\gt{h}[t,t^{-1}]\oplus \CC {\bf c}$, for $t$ a formal variable, and define a Lie algebra structure on $\hat{\gt{h}}$ by declaring that ${\bf c}$ is central, and $[u\otimes t^m,v\otimes t^n]=m\lab u,v\rab\delta_{m+n,0}\,{\bf c}$ for $u,v\in\gt{h}$ and $m,n\in\ZZ$. We follow tradition and write $u(m)$ as a shorthand for $u\otimes t^m$. The Lie algebra $\hat{\gt{h}}$ has a triangular decomposition $\hat{\gt{h}}=\hat{\gt{h}}^-\oplus \hat{\gt{h}}^0\oplus \hat{\gt{h}}^+$ where $\hat{\gt{h}}^{\pm}:=\gt{h}[t^{\pm 1}]t^{\pm 1}$ and $\hat{\gt{h}}^0:=\gt{h}\oplus \CC{\bf c}$. 

We require a bilinear function $b:L\times L\to \ZZ/2\ZZ$ with the property that $b(\lambda,\mu)+b(\mu,\lambda)=\lab \lambda,\mu\rab+\lab\lambda,\lambda\rab\lab\mu,\mu\rab +2\ZZ$. If $\{\varepsilon_i\}$ is an ordered $\ZZ$-basis for $L$ then we may take $b$ to be the unique such function for which
\begin{gather}
	b(\le_i,\le_j)=\begin{cases}
		0&\text{when $i\leq j$,}\\
		1&\text{when $i>j$.}
		\end{cases}
\end{gather}
Set $\beta(\lambda,\mu):=(-1)^{b(\lambda,\mu)}$, and define $\CC_{\beta}[L]$ to be the ring generated by symbols $\vac_{\lambda}$ for $\lambda\in L$ subject to the relations $\vac_{\lambda}\vac_{\mu}=\beta(\lambda,\mu)\vac_{\lambda+\mu}$. 

\begin{rmk}
The algebra $\CC_{\beta}[L]$ is isomorphic to the quotient $\CC[\hat{L}]/\lab \kappa+1\rab$, where $\hat{L}$ is the unique (up to isomorphism) central extension of $L$ by $\lab \kappa\rab\simeq\ZZ/2\ZZ$ such that 
\begin{gather}
aa'=\kappa^{\lab \bar{a},\bar{a}'\rab+\lab \bar{a},\bar{a}\rab\lab \bar{a}',\bar{a}'\rab}a'a,
\end{gather} 
for $a,a'\in \hat{L}$ lying above $\bar{a},\bar{a}'\in L$, respectively. (Cf. \cite{FLM}.)
\end{rmk}

Now define a $\hat{\gt{h}}^0\oplus \hat{\gt{h}}^+$-module structure on $\CC_{\beta}[L]$ by setting ${\bf c}\vac_{\lambda}=\vac_{\lambda}$ and $u(m)\vac_{\lambda}= \delta_{m,0} \lab u,\lambda\rab\vac_{\lambda}$ for $u\in \gt{h}$ and $\lambda\in L$, and define $V_L$ to be the induced $\hat{\gt{h}}$-module, 
\begin{gather}
	V_L:=U(\hat{\gt{h}})\otimes_{U(\hat{\gt{h}}^0\oplus \hat{\gt{h}}^+)}\CC_{\beta}[L].
\end{gather}
Then, according to \S5.4.2 of \cite{MR2082709}, for example, $V_L$ admits a unique super vertex algebra structure $Y:V_L\to (\End V_L)[[z,z^{-1}]]$ such that $1\otimes \vac_{0}$ is the vacuum vector, 
\begin{gather}\label{eqn:va:latva-Yu}
	Y(u(-1)\otimes \vac_{0},z)=\sum_{n\in\ZZ} u(n)z^{-n-1}
\end{gather}
for $u\in \gt{h}$, and 
\begin{gather}\label{eqn:va:latva-Ylambda}
	Y(1\otimes \vac_{\lambda},z)=
		\exp\left(-\sum_{n<0}\frac{\lambda(n)}{n}z^{-n}\right)
		\exp\left(-\sum_{n>0}\frac{\lambda(n)}{n}z^{-n}\right)
		\vac_{\lambda}z^{\lambda(0)}
\end{gather}
for $\lambda \in L$. 
Here $\vac_{\lambda}$ denotes the operator $p\otimes \vac_{\mu}\mapsto \beta(\lambda,\mu)p\otimes \vac_{\lambda+\mu}$, and $z^{\lambda(0)}(p\otimes \vac_{\mu}):=(p\otimes \vac_{\mu}) z^{\lab \lambda,\mu\rab}$. Note that we have 
\begin{gather}
	V_L\simeq S(\hat{\gt{h}}^-)\otimes \CC[L]
\end{gather}
as modules for $\hat{\gt{h}}^-\oplus \hat{\gt{h}}^0$.

Given that $\{\le_i\}$ is a basis for $L$, choose $\le_i'\in L\otimes_{\ZZ}\QQ$ such that $\lab \le_i',\le_j\rab=\delta_{i,j}$, and define 
\begin{gather}
	\omega:=\frac{1}{2}\sum_{i=1}^3\le_i'(-1)\le_i(-1)\otimes \vac_0. 
\end{gather}
Then $\omega$ is a conformal element for $V_L$ with central charge equal to the rank of $L$. 
If we define $L(n)\in \End V_L$ so that $Y(\omega,z)=\sum_{n\in \ZZ}L(n)z^{-n-2}$ then $[L(0),v(n)]=-nv(n)$ and 
$1\otimes \vac_{\lambda}$ is an eigenvector for $L(0)$ with eigenvalue $\lab \lambda,\lambda\rab/2$. Note that we do not assume that the bilinear form on $L$ is positive-definite. Vectors of non-positive length in $L$ will give rise to infinite dimensional eigenspaces for $L(0)$, so in general $(V_L,Y,\vac_0,\omega_u)$ is a conformal super vertex algebra, but not a super vertex operator algebra.

Automorphisms of $L$ can be lifted to automorphisms of $V_L$. For suppose given $g\in \Aut(L)$ and a function $\alpha:L\to \{\pm 1\}$ satisfying
\begin{gather}\label{eqn:va:latva-alphacond}
	\alpha(\lambda+\mu)\beta(\lambda,\mu)=\alpha(\lambda)\alpha(\mu)\beta(g\lambda,g\mu)
\end{gather}
for $\lambda,\mu\in L$. Then we obtain an automorphism $\hat{g}$ of $\Aut(V_L)$ by setting
\begin{gather}\label{eqn:va:latva-hatg}
	\hat{g}(p\otimes \vac_\lambda):=\alpha(\lambda) (g\cdot p)\otimes\vac_{g\lambda},
\end{gather}
for $p\in S(\hat{\gt{h}}^-)$ and $\lambda\in L$, where $g\cdot p$ denotes the natural extension of the action of $\Aut(L)$ on $\gt{h}=L\otimes_{\ZZ}\CC$ to $S(\hat{\gt{h}}^-)$, determined by $g\cdot u(m)=(gu)(m)$ for $u\in\gt{h}$. 

For example, take 
$g$ to be the {\em Kummer involution} of $L$, given by $g\lambda=-\lambda$ for $\lambda\in L$. Then $\beta(\lambda,\mu)=\beta(-\lambda,-\mu)$ for all $\lambda,\mu\in L$, since $\beta$ is bi-multiplicative, so we may take $\alpha\equiv 1$ in (\ref{eqn:va:latva-alphacond}). We denote the corresponding automorphism of $V_L$ by $\theta$, and note that the action of $\theta$ on $V_L$ is given explicitly as follows. If $p\in S(\hat{\gt{h}}^-)$ is a homogeneous polynomial of degree $k$ in variables $u_i(-m_i)$, where $u_i\in\gt{h}$ and the $m_i$ are positive integers, then 
\begin{gather}\label{eqn:va:latva-theta}
	\theta(p\otimes v_{\lambda})=(-1)^k p\otimes v_{-\lambda}.
\end{gather}

\subsection{Lattice Vertex Algebra Modules}\label{sec:va:latvamod}

Let $\gamma$ be an element of the dual lattice $L^*:=\{\lambda\in L\otimes_{\ZZ}\QQ\mid \lab \lambda,L\rab\subset\ZZ\}$. Define $\CC_{\beta}[L+\gamma]$ to be the complex vector space generated by symbols $\vac_{\mu+\gamma}$ for $\mu\in L$,
regarded as an $\CC_{\beta}[L]$-module according to the rule $\vac_{\lambda}\cdot\vac_{\mu+\gamma}=\beta(\lambda,\mu)\vac_{\lambda+\mu+\gamma}$. Equip $\CC_{\beta}[L+\gamma]$ with an $U(\hat{\gt{h}}^0\oplus \hat{\gt{h}}^+)$-module structure much as before, by letting ${\bf c}\vac_{\mu+\gamma}=\vac_{\mu+\gamma}$ and $u(m)\vac_{\mu+\gamma}=\delta_{m,0}\lab u,\mu+\gamma\rab \vac_{\mu+\gamma}$ for $u\in\gt{h}$ and $\mu\in L$. Let $V_{L+\gamma}$ be the $\hat{\gt{h}}$-module defined by setting $V_{L+\gamma}:=U(\hat{\gt{h}})\otimes_{U(\hat{\gt{h}}^0\oplus \hat{\gt{h}}^+)}\CC_{\beta}[L+\gamma]$. Then we have an isomorphism
\begin{gather}
	V_{L+\gamma}\simeq S(\hat{\gt{h}}^-)\otimes \CC[L+\gamma]
\end{gather}
of modules for $\hat{\gt{h}}^-$. Define vertex operators $Y_{\gamma}:V_L\to(\End V_{L+\gamma})[[z,z^{-1}]]$ using the same formulas as before but interpret the operator $\vac_{\lambda}$ in (\ref{eqn:va:latva-Ylambda}) as $\vac_{\lambda}(p\otimes \vac_{\mu+\gamma})=\beta(\lambda,\mu)p\otimes \vac_{\lambda+\mu+\gamma}$, according to the $\CC_{\beta}[L]$-module structure on $\CC_{\beta}[L+\gamma]$ prescribed above. Note that the construction of $V_{L+\gamma}$ depends upon the choice of coset representative $\gamma\in L^*$, so that $V_{L+\gamma}$ might be different from $V_{L+\gamma'}$, as a $\CC_{\beta}[L]$-module, for example, even when $L+\gamma=L+\gamma'$, but different choices of coset representative are guaranteed to define isomorphic $V_L$-modules according to \cite{MR1245855}.

The 
construction just described may be generalized so as to realize certain twisted modules for $V_L$. We give a brief description here, and refer to \S3 of \cite{MR1284796} for more details. 

Choose a vector $h\in\gt{h}$. Then for $p\in  S(\hat{\gt{h}}^-)$ and $\lambda\in L$ we have $h(0)p\otimes \vac_{\lambda}=\lab h,\lambda\rab p\otimes\vac_{\lambda}$. So if $h$ is chosen to lie in 
$L\otimes_{\ZZ}\QQ$ then 
\begin{gather}\label{eqn:va:latvamod-sigmah}
g_h:=e^{2\pi i h(0)}
\end{gather} 
is a finite order automorphism of $V_L$, which acts as multiplication by $e^{2\pi i \lab h,\lambda\rab}$ on the vector $p\otimes \vac_{\lambda}$. The kernel of the map $L\otimes_{\ZZ}\QQ\to \Aut(V_L)$ given by $h\mapsto g_h$ is exactly $L^*$, so $(L\otimes_{\ZZ}\QQ)/L^*$ is naturally a group of automorphisms of $V_L$. We may construct all the corresponding twisted modules for $V_L$ explicitly.

To do this choose an $h$ in $L\otimes_{\ZZ}\QQ$ and let $\CC[L+h]$ be the complex vector space generated by symbols $\vac_{\lambda+h}$ for $\lambda\in L$. Just as before, we define a $U(\hat{\gt{h}}^0\oplus \hat{\gt{h}}^+)$-module structure on $\CC[L+h]$ by setting ${\bf c}\vac_{\mu}=\vac_{\mu}$ and $u(m)\vac_{\mu}=\delta_{m,0}\lab u,\mu\rab \vac_{\mu}$ for $u\in\gt{h}$ and $\mu\in L+h$. Let $V_{L+h}$ be the $\hat{\gt{h}}$-module defined by setting $V_{L+h}:=U(\hat{\gt{h}})\otimes_{U(\hat{\gt{h}}^0\oplus \hat{\gt{h}}^+)}\CC[L+h]$, so that we have an isomorphism
\begin{gather}
	V_{L+h}\simeq S(\hat{\gt{h}}^-)\otimes \CC[L+h]
\end{gather}
of modules for $\hat{\gt{h}}^-$. Taking $M$ to be a positive integer such that $Mh\in L^*$, define vertex operators $Y_h:V_L\to(\End V_{L+h})[[z^{1/M},z^{-1/M}]]$ using the same formulas as before but interpret the operator $\vac_{\lambda}$ in (\ref{eqn:va:latva-Ylambda}) as $\vac_{\lambda}(p\otimes \vac_{\mu+h})=\beta(\lambda,\mu)p\otimes \vac_{\lambda+\mu+h}$. Then $V_{L+h}=(V_{L+h},Y_{h})$ is an irreducible $g_h$-twisted module for $V_L$, and any $g_h$-twisted module for $V_L$ is of the form $V_{L+h'}$ for some $h'\in L\otimes_{\ZZ}\QQ$ that is congruent to $h$ modulo $L^*$.

Note that the action of $L\otimes_{\ZZ}\QQ$ on $V_L$, given by $h\mapsto g_{h}$, extends to the $g_{h'}$-twisted module $V_{L+h'}$, for $h'\in L\otimes_{\ZZ}\QQ$. For given $h,h'\in L\otimes_{\ZZ}\QQ$, we may define
\begin{gather}\label{eqn:va:latvamod-sigmahtw}
	g_{h}(p\otimes \vac_{\lambda+h'}):=e^{2\pi i \lab h,\lambda\rab}(p\otimes \vac_{\lambda+h'})
\end{gather}
for $p\in S(\hat{\gt{h}}^-)$ and $\lambda\in L$. Then we have $g_h Y_{h'}(a,z)b=Y_{h'}(g_h a,z)g_h b$ for $a\in V_L$ and $b\in V_{L+h'}$.

\subsection{Clifford Module Vertex Algebra}\label{sec:va:cliffmod}

We also require the standard procedure---see \cite{MR1123265} for a general treatment, and \cite{MR1372717} for the special, one-dimensional case we consider here---which attaches a Clifford module super vertex operator algebra to a vector space equipped with a symmetric bilinear form. 

So let $\gt{p}$ 
be a one dimensional complex vector space equipped with a non-degenerate symmetric bilinear form $\lab\cdot\,,\cdot\rab$. Set $\hat{\gt{p}}=\gt{p}[t,t^{-1}]t^{1/2}$ and write $a(r)$ for $a\otimes t^r$. Extend the bilinear form from $\gt{p}$ to $\hat{\gt{p}}$ by requiring that $\lab a(r),b(s)\rab=\lab a,b\rab\delta_{r+s,0}$. 
Set $\hat{\gt{p}}^{\pm}=\gt{p}[t^{\pm 1}]t^{\pm 1/2}$, write $\lab\hat{\gt{p}}^{\pm}\rab$ for the sub algebra of $\Cliff(\hat{\gt{p}})$ generated by $\hat{\gt{p}}^{\pm}$ and define a one-dimensional $\lab \hat{\gt{p}}^+\rab$-module $\CC\vac$ by requiring that ${\bf 1}\vac=\vac$ and $a(r)\vac=0$ for $a\in \gt{p}$ and $r>0$. Here $\Cliff(\hat{\gt{p}})$ denotes the {\em Clifford algebra} attached to $\hat{\gt{p}}$, which we take to be the quotient of the tensor algebra $T(\hat{\gt{p}})=\CC{\bf 1}\oplus \hat{\gt{p}}\oplus \hat{\gt{p}}^{\otimes 2}\oplus \cdots$ by the ideal generated by expressions of the form $u\otimes u+\frac{1}{2}\lab u,u\rab{\bf 1}$ for $u\in \hat{\gt{p}}$.

Observe that the induced $\Cliff(\hat{\gt{p}})$-module, $A(\gt{p})=\Cliff(\hat{\gt{p}})\otimes_{\lab \hat{\gt{p}}^+\rab}\CC\vac$, is isomorphic to $\bigwedge(\hat{\gt{p}}^-)\vac$ as a $\lab \hat{\gt{p}}^-\rab$-module. 
We obtain a super vertex algebra structure on $A(\gt{p})$ 
by setting 
\begin{gather}
Y(a(-1/2)\vac,z)=\sum_{n\in\ZZ}a(n+1/2)z^{-n-1}
\end{gather}
for $a\in \gt{p}$, 
for the reconstruction theorem of \cite{MR2082709} ensures that this rule extends uniquely to a super vertex algebra structure $Y:A(\gt{p})\otimes A(\gt{p})\to A(\gt{p})((z))$ with $Y(\vac,z)=\Id$. 

Let $p\in\gt{p}$ such that $\lab p,p\rab=-2$. 
We obtain a super vertex operator algebra structure, with central charge $c=1/2$, by taking
\begin{gather}
	\omega=\frac{1}{4}p(-3/2)p(-1/2)\vac
\end{gather}
to be the Virasoro element.

To construct canonically-twisted modules for $A(\gt{p})$ set $\hat{\gt{p}}_{\tw}=\gt{p}[t,t^{-1}]$ and extend the bilinear form from $\gt{p}$ to $\hat{\gt{p}}_{\tw}$ as before by requiring that $\lab a(r),b(s)\rab=\lab a,b\rab\delta_{r+s,0}$. Set $\hat{\gt{p}}_{\tw}^{>}=\gt{p}[t]t$ and $\hat{\gt{p}}_{\tw}^{\leq}=\gt{p}[t^{-1}]$, and define a $1$-dimensional $\lab \hat{\gt{p}}_{\tw}^>\rab$-module $\CC\vac_{\tw}$ by requiring, much as before, that ${\bf 1}\vac_{\tw}=\vac_{\tw}$ and $a(r)\vac=0$ for $a\in \gt{p}$ and $r>0$. Then for the induced $\Cliff(\hat{\gt{p}}_{\tw})$-module, 
\begin{gather}
A(\gt{p})_{\tw}:=\Cliff(\hat{\gt{p}}_{\tw})\otimes_{\lab \hat{\gt{p}}_{\tw}^>\rab}\CC\vac_{\tw},
\end{gather}
there is a unique linear map $Y_{\tw}:A(\gt{p})\otimes A(\gt{p})_{\tw}\to A(\gt{p})_{\tw}((z^{1/2}))$ such that 
\begin{gather}\label{eqn:Ytwu}
Y_{\tw}(u(-1/2)\vac,z)=\sum_{n\in \ZZ}u(n)z^{-n-1/2}
\end{gather}
for $u\in\gt{p}$, and $(A(\gt{p})_{\tw},Y_{\tw})$ is a canonically-twisted module for $A(\gt{p})$. Again one may use (a suitably modified formulation of) the reconstruction theorem of \cite{MR2082709} to see this (cf. \cite{MR2074176}). We refer to \cite{MR1372717} for a concrete and detailed description of $Y_{\tw}$. 
Note that $A(\gt{p})$ is isomorphic to $\bigwedge(\hat{\gt{p}}_{\tw}^{\leq})\vac$ as a $\lab \hat{\gt{p}}_{\tw}^{\leq}\rab$-module.

With $p\in\gt{p}$ as above, such that $\lab p,p\rab=-2$, we have $p(0)^2={\bf 1}$ in $\Cliff(\gt{p})$. Set 
\begin{gather}
\vac_\tw^\pm:=({\bf 1}\pm p(0))\vac_\tw,
\end{gather} 
so that $p(0)\vac_\tw^\pm=\pm\vac_\tw^\pm$. Then $A(\gt{p})_\tw=A(\gt{p})_\tw^+\oplus A(\gt{p})_\tw^-$ is a decomposition of $A(\gt{p})_\tw$ into irreducible canonically-twisted $A(\gt{p})$-modules, where $A(\gt{p})_\tw^\pm$ denotes the sub module of $A(\gt{p})_\tw$ generated by $\vac_\tw^\pm$.
\begin{gather}
A(\gt{p})_{\tw}^\pm:=\Cliff(\hat{\gt{p}}_{\tw})\otimes_{\lab \hat{\gt{p}}_{\tw}^>\rab}\CC\vac_{\tw}^\pm
\end{gather}
From (\ref{eqn:Ytwu}) we see that the $L(0)$-degree preserving component of $Y_\tw(p(-1/2)\vac,z)$ is $p(0)$. Computing the graded-trace of $p(0)$ on $A(\gt{p})_\tw^{\pm}$, we find
\begin{gather}\label{eqn:va:cliffmod-trpAtw}
	\tr_{A(\gt{p})_\tw^{\pm}}p(0)q^{L(0)-c/24}=\pm q^{1/24}\prod_{n>0}(1-q^n),
\end{gather}
where the factor $q^{1/24}$ appears because $L(0)\vac_\tw^{\pm}=\frac{1}{16}\vac_\tw^{\pm}$ and $c=1/2$.

\subsection{Cone Vertex Algebra}\label{sec:va:cva}

Let $L$ be an integral lattice as before, and suppose $\{\le_i\}$ is a $\ZZ$-basis for $L$. 
Define $P$ to be the monoid of non-negative rational combinations of the chosen basis vectors $\le_i$, 
\begin{gather}
	P:=\left\{\sum_i \alpha_i\le_i\in L\otimes_{\ZZ}\QQ\mid \alpha_i\geq 0,\,\forall i\right\},
\end{gather}
and define $N$ to be the semigroup of strictly negative rational combinations of the $\le_i$,
\begin{gather}
	N:=\left\{\sum_i \alpha_i\le_i\in L\otimes_{\ZZ}\QQ\mid \alpha_i< 0,\,\forall i \right\}.
\end{gather}
Define $D:=P\cup N$ to be the union of $P$ and $N$. 
Our goal in this section is to attach a vertex algebra structure to the intersection $D\cap L$. For convenience we use the abbreviated notation $D(L):=D\cap L$, and more generally 
\begin{gather}\label{eqn:va:cva-DLgamma}
D(L+\gamma):=D\cap (L+\gamma)
\end{gather} 
for $\gamma\in L\otimes_{\ZZ}\QQ$. We interpret the notations $P(L+\gamma)$ and $N(L+\gamma)$ similarly.

Given $K\subset L$ write $V_K$ for the $\hat{\gt{h}}$-submodule of $V_L$ generated by the $\vac_{\lambda}$ for $\lambda\in K$,
\begin{gather}
	V_K\simeq S(\hat{\gt{h}}^-)\otimes \CC[K].
\end{gather}
Observe that if $K\subset L$ is closed under addition and contains $0$---i.e. if $K$ is a submonoid of $L$---then $V_K$ is a sub super vertex algebra of $V_L$, and $\omega$ is a conformal element for $V_K$. Furthermore, if $K'\subset L$ satisfies $K+K'\subset K'$ then the restriction of the vertex operators $a\otimes b\mapsto Y(a,z)b$ 
to $V_{K}\otimes V_{K'}<V_L\otimes V_L$ equips $V_{K'}$ with a module structure over $V_K$. 

So in particular, $V_{P(L)}$ (cf. (\ref{eqn:va:cva-DLgamma})) is a conformal super vertex algebra. If the basis $\{\le_i\}$ is chosen so that $P$ has no non-trivial vectors with non-positive length squared, then the eigenspaces for the action of $L(0)$ on $V_{P(L)}$ are finite-dimensional, the eigenvalues of $L(0)$ are contained in $\frac{1}{2}\ZZ$ and bounded from below, and thus $V_{P(L)}$ is a super vertex operator algebra. 

We will now show that the super vertex algebra structure on $V_{P(L)}$ extends naturally to $V_{D(L)}=V_{P(L)}\oplus V_{N(L)}$. 
For this we require a $V_{P(L)}$-module structure on $V_{N(L)}$, which we achieve by implementing the following standard method (cf. e.g. \S2 of \cite{MR1284796}).

Suppose that $g$ is an automorphism of a super vertex algebra $V=(V,Y,\vac)$ and, $g_M\in \GL(M)$ is a linear automorphism of a $V$-module $M=(M,Y_M)$ satisfying $g_MY_M(a,z)m=Y_M(ga,z)g_Mm$ for $a\in V$ and $m\in M$. Observe then that we obtain a new $V$-module structure $M^g:=(M,Y_M^g)$ on the vector space underlying $M$ by setting 
\begin{gather}\label{eqn:va:cva-Ymg}
Y_M^g(a,z)m:=g_MY_M(a,z)g_M^{-1}m
\end{gather}
for $a\in V$ and $m\in M$. Indeed, we have $Y_M^g(a,z)=Y_M(ga,z)$. 

Now take $M=V=V_L$ and $g=g_M=\theta$ in (\ref{eqn:va:cva-Ymg}), where $\theta\in \Aut(V_L)$ is the involution defined in \S\ref{sec:va:latva}, determined by requiring that $\theta(1\otimes \vac_{\lambda})=1\otimes \vac_{-\lambda}$ for $\lambda\in L$, and $[\theta,u(m)]=-u(m)$ for $u\in \gt{h}$ (cf. (\ref{eqn:va:latva-theta})). Observe that $\theta$ maps $V_{N(L)}$ to $V_{(-N)\cap L}$ which is a subspace of $V_{P(L)}$. Since 
\begin{gather}
P+(-N)=\{\lambda+\mu\mid \lambda\in P,\,\mu\in -N\}
\end{gather}
is a subset of $-N$, 
the space $V_{(-N)\cap L}$ is even a $V_{P(L)}$-submodule of $V_{P(L)}$, so we obtain a $V_{P(L)}$-module structure on $V_{N(L)}$ by restricting the map $a\otimes b\mapsto Y^{\theta}(a,z)b$ to $V_{P(L)}\otimes V_{N(L)}$. Note that $Y^{\theta}(a,z)b=\theta Y(a,z)\theta b=Y(\theta a,z)b$.

For a vertex algebra structure on $V_{D(L)}$ we must also identify maps $V_{N(L)}\otimes V_{P(L)}\to V_{N(L)}((z))$ and $V_{N(L)}\otimes V_{N(L)}\to V_{P(L)}((z))$. For the first of these we use $\widetilde{Y}(a,z)b:=Y(a,z)\theta b$. For the second we set $\widetilde{Y}(a,z)b:=\theta Y(a,z)b=Y(\theta a, z)\theta b$. To summarize, we define a vertex operator correspondence $\widetilde{Y}:V_{D(L)}\otimes V_{D(L)}\to V_{D(L)}((z))$, by setting
\begin{gather}\label{eqn:va:cva-vdvops}
	\widetilde{Y}(a,z)b
	:=\begin{cases}
	Y(a,z)b,&\text{ for $a,b\in V_{P(L)}$,}\\
	Y(\theta a,z)b,&\text{ for $a\in V_{P(L)}$ and $b\in V_{N(L)}$,}\\
	Y(a,z)\theta b,&\text{ for $a\in V_{N(L)}$ and $b\in V_{P(L)}$,}\\
	\theta Y(a,z)b,&\text{ for $a,b\in V_{N(L)}$,}\\
	\end{cases}
\end{gather}
where $Y$ denotes the usual vertex operator correspondence on $V_L$, determined by (\ref{eqn:va:latva-Yu}) and (\ref{eqn:va:latva-Ylambda}).

\begin{thm}\label{thm:va:cva-VD}
The four-tuple $(V_{D(L)},\widetilde{Y},\vac,\omega)$ is a conformal super vertex algebra. It is a super vertex operator algebra if $D$ has no non-trivial vectors of non-positive length squared.
\end{thm}
\begin{proof}
The proof is a standard exercise in lattice vertex algebra computations. The fundamental reason that the construction works is the fact that we obtain a commutative monoid structure $\tilde{+}$ 
on $D$ when we define
\begin{gather}
	\lambda\tilde{+}\mu
	:=\begin{cases}
	\lambda+\mu,&\text{ for $\lambda,\mu\in P$,}\\
	-\lambda+\mu,&\text{ for $\lambda\in P$ and $\mu\in N$,}\\
	\lambda-\mu,&\text{ for $\lambda\in N$ and $\mu\in P$,}\\
	-\lambda-\mu,&\text{ for $\lambda,\mu\in N$.}\\
	\end{cases}
\end{gather}
The remaining details are left to the reader.
\end{proof}

Observe that the decomposition $V_{D(L)}=V_{P(L)}\oplus V_{N(L)}$ determines a $\ZZ/2$-grading on $V_{D(L)}$. We call this the {\em sign grading}, and we define the {\em sign automorphism} of $V_{D(L)}$ to be the linear map $s:V_{D(L)}\to V_{D(L)}$ determined by setting
\begin{gather}\label{eqn:va:cva-signaut}
	s(a):=\begin{cases}
		a,&\text{ when $a\in V_{P(L)}$,}\\
		-a,&\text{ when $a\in V_{N(L)}$.}
		\end{cases}
\end{gather}
It follows easily from the definition (\ref{eqn:va:cva-vdvops}) of the super vertex algebra structure on $V_{D(L)}$ that $s$ is indeed an automorphism of $V_{D(L)}$.

We can 
construct certain twisted (and untwisted) modules for $V_{D(L)}$, by suitably modifying the constructions recalled in \S\ref{sec:va:latvamod}. Namely, for $h\in L\otimes_{\ZZ}\QQ$ take $V_{D(L+h)}$ to be the $\hat{\gt{h}}$-module defined by setting $V_{D(L+h)}:=U(\hat{\gt{h}})\otimes_{U(\hat{\gt{h}}^0\oplus \hat{\gt{h}}^+)}\CC[D(L+h)]$, where $\CC[D(L+h)]$ is the complex vector space generated by symbols $\vac_{\mu}$ for $\mu\in D\cap(L+h)$, regarded as a $U(\hat{\gt{h}}^0\oplus \hat{\gt{h}}^+)$-module by setting ${\bf c}\vac_{\mu}=\vac_{\mu}$ and $u(m)\vac_{\mu}=\delta_{m,0}\lab u,\mu\rab \vac_{\mu}$ for $u\in\gt{h}$ and $\mu\in D\cap(L+h)$. As usual, we have an isomorphism
\begin{gather}
	V_{D(L+h)}\simeq S(\hat{\gt{h}}^-)\otimes \CC[D(L+h)]
\end{gather}
of modules for $\hat{\gt{h}}^-$. Taking $M$ to be a positive integer such that $Mh\in L^*$, define vertex operators $\widetilde{Y}_h:V_{D(L)}\to(\End V_{D(L+h)})[[z^{1/M},z^{-1/M}]]$ using (\ref{eqn:va:latva-Yu}), (\ref{eqn:va:latva-Ylambda}) and (\ref{eqn:va:cva-vdvops}), but interpret the operator $\vac_{\lambda}$ in (\ref{eqn:va:latva-Ylambda}) as $\vac_{\lambda}(p\otimes \vac_{\mu+h})=\beta(\lambda,\mu)p\otimes \vac_{\lambda+\mu+h}$. 
\begin{thm}\label{thm:va:cva-VDh}
Let $h\in L\otimes_{\ZZ}\QQ$. Then the pair $(V_{D(L+h)},\widetilde{Y}_h)$ is a $g_h$-twisted module for $V_{D(L)}$. In particular, $(V_{D(L+h)},\widetilde{Y}_h)$ is a $V_{D(L)}$-module when $h\in L^*$.
\end{thm}

\subsection{Main Construction}\label{sec:va:cnstn}

We now take $L=\ZZ \le_1+\ZZ \le_2+\ZZ \le_3$ to be the rank $3$ lattice with bilinear form $\lab\cdot\,,\cdot\rab$ determined by 
\begin{gather}\label{eqn:va:cnstn-bilinearform}
\lab \le_i,\le_j\rab=2-\delta_{i,j}.
\end{gather} 
Then $L$ is an integral, non-even lattice with signature $(1,2)$. Set $\rho:=(\le_1+\le_2+\le_3)/5$ and observe that 
\begin{gather}\label{eqn:va:cnstn-iplambdarho}
\lab \lambda,\rho\rab=k+l+m
\end{gather}
for $\lambda=k\le_1+l\le_2+m\le_3$, so $\rho$ belongs to the dual $L^*$ of $L$. In fact, $L^*/L$ is cyclic of order $5$, and $\rho+L$ is a generator. If we set 
\begin{gather}\label{eqn:va:cnstn-Lj}
	L^j:=\{\lambda\in L\mid \lab\lambda,\rho\rab=j\jmod{2}\},
\end{gather}
then $L=L^0\cup L^1$ is the decomposition of $L$ into its even and odd parts, by which we mean that $\lab\lambda,\lambda\rab$ is even or odd according as $\lambda$ lies in $L^0$ or $L^1$.

Let $V_L$ be the super vertex operator algebra attached to $L$ via the construction of \S\ref{sec:va:latva}, where the bilinear function $b:L\times L\to \ZZ/2\ZZ$ is determined by setting
\begin{gather}\label{eqn:va:cnstn-b}
	b(\le_i,\le_j):=\begin{cases}
		0&\text{when $i\leq j$,}\\
		1&\text{when $i>j$.}
		\end{cases}
\end{gather}

There is an obvious action of the symmetric group $S_3$ on $L$, by permutations of the basis vectors $\le_i$. We lift this action to $V_L$ in the following way. Recall from \S\ref{sec:va:latva} that a lift $\hat{g}\in \Aut(V_L)$ of an automorphism $g\in \Aut(L)$ is determined by a choice of function $\alpha:L\to \{\pm 1\}$ satisfying (\ref{eqn:va:latva-alphacond}). Observe that any such automorphism $\hat{g}$ restricts to an automorphism of $V_{D(L)}$, so long as $g$ preserves the subset $D(L)\subset L$. Taking $\mu=k\lambda$ in (\ref{eqn:va:latva-alphacond}) we have $\alpha((k+1)\lambda)=\alpha(\lambda)\alpha(k\lambda)\beta(\lambda,\lambda)^k\beta(g\lambda, g\lambda)^k$, since $\beta$ is bi-mulitplicative, so given the prescription (\ref{eqn:va:cnstn-b}) we see that $\beta(\lambda,\lambda)=k_1k_2+k_2k_3+k_3k_1$ for $\lambda=k_1\le_1+k_2\le_2+k_3\le_3$, 
which is invariant under the action of $S_3$. So actually $\beta(\lambda,\lambda)=\beta(g\lambda,g\lambda)$, and thus we may assume $\alpha(k\lambda)=\alpha(\lambda)^k$ in (\ref{eqn:va:latva-alphacond}) for $\lambda\in L$ and $k$ a positive integer, when $g$ acts by permuting the $\le_i$. Observe also that for $\lambda,\mu,\nu\in L$ we have
\begin{gather}
	\alpha(\lambda+\mu+\nu)\beta(\lambda,\mu)\beta(\mu,\nu)\beta(\nu,\lambda)
	=\alpha(\lambda)\alpha(\mu)\alpha(\nu)\beta(g\lambda,g\mu)\beta(g\mu,g\nu)\beta(g\nu,g\lambda)
\end{gather}
according to (\ref{eqn:va:latva-alphacond}), which specializes to
\begin{gather}
	\begin{split}\label{eqn:va:cnstn-alphavareps}
	&\alpha(\lambda)\beta(\le_1,\le_2)^{k_1k_2}\beta(\le_2,\le_3)^{k_2k_3}\beta(\le_3,\le_1)^{k_3k_1}\\
	&=\alpha(\le_1)^{k_1}\alpha(\le_2)^{k_2}\alpha(\le_3)^{k_3}
	\beta(g\le_1,g\le_2)^{k_1k_2}\beta(g\le_2,g\le_3)^{k_2k_3}\beta(g\le_3,g\le_1)^{k_3k_1}
	\end{split}
\end{gather}
for $\lambda=k_1\le_1+k_2\le_2+k_3\le_3$. 

Consider the case that $g=\sigma$ is the cyclic permutation $(123)$. From (\ref{eqn:va:cnstn-alphavareps}) we see that we may lift $\sigma$ to $\Aut(V_L)$ by taking $\alpha(\le_i)=1$ for $i\in\{1,2,3\}$, and more generally $\alpha(k_1\le_1+k_2\le_2+k_3\le_3)=(-1)^{k_2k_3+k_3k_1}$, 
in the construction of \S\ref{sec:va:latva}. We denote the corresponding automorphism of $V_L$ by $\hat{\sigma}_0$.
\begin{gather}\label{eqn:va:cnstn-sigma}
	\hat{\sigma}_0(p\otimes \vac_{k_1\le_1+k_2\le_2+k_3\le_3})
	:=(-1)^{k_2k_3+k_3k_1}(\sigma\cdot p)\otimes\vac_{k_3\le_1+k_1\le_2+k_2\le_3}
\end{gather}
Next consider $g=\tau:=(12)$. Applying (\ref{eqn:va:cnstn-alphavareps}) we see that we may lift $\tau$ to $\Aut(V_L)$ by taking $\alpha(\le_i)=1$ as before, and more generally $\alpha(k_1\le_1+k_2\le_2+k_3\le_3)=(-1)^{k_1k_2}$, 
in the construction of \S\ref{sec:va:latva}. We denote the corresponding automorphism of $V_L$ by $\hat{\tau}_0$.
\begin{gather}\label{eqn:va:cnstn-tau}
	\hat{\tau}_0(p\otimes \vac_{k_1\le_1+k_2\le_2+k_3\le_3})
	:=(-1)^{k_1k_2}(\tau\cdot p)\otimes\vac_{k_2\le_1+k_1\le_2+k_3\le_3}
\end{gather}
Using (\ref{eqn:va:cnstn-sigma}) and (\ref{eqn:va:cnstn-tau}) one can 
check that $\hat{\sigma}_0^3=\hat{\tau}_0^2=(\hat{\tau}_0\hat{\sigma}_0)^2=\Id$ in $\Aut(V_L)$, so $\hat{\sigma}_0$ and $\hat{\tau}_0$ do indeed generate a copy of $S_3$ in $\Aut(V_L)$.

Observe that $V_L=V_{L^0}\oplus V_{L^1}$ is the decomposition of $V_L$ into its even and odd parity subspaces, where $L^j$ is defined by (\ref{eqn:va:cnstn-Lj}). Recall the automorphisms $g_h$ of $V_L$, defined for $h\in L\otimes_{\ZZ}\QQ$ by (\ref{eqn:va:latvamod-sigmah}). Then we see from (\ref{eqn:va:cnstn-Lj}) that the canonical involution of $V_L$, acting as $+1$ on the even subspace $V_{L^0}$, and $-1$ on the odd subspace $V_{L^1}$, is realized by $g_{\rho/2}$. 
So the canonically-twisted modules for $V_L$ are exactly the $V_{L+a\rho/2}$, for $a\in\{1,3,5,7,9\}$ (cf. \S\ref{sec:va:latvamod}). 

The prescription (\ref{eqn:va:latvamod-sigmahtw}) furnishes an extension of the action of the canonical involution $g_{\rho/2}$, from $V_L$ to $V_{L+a\rho/2}$. Since $\rho$ is $S_3$-invariant we may also extend the actions of $\hat\sigma_0$ and $\hat\tau_0$ to $V_{L+a\rho/2}$, by setting
\begin{gather}
	\begin{split}\label{eqn:va:cnstn-sigtaunoughtVLtw}
	\hat\sigma_0(p\otimes \vac_{\lambda+a\rho/2})&:=(-1)^{k_2k_3+k_3k_1}(\sigma\cdot p)\otimes\vac_{\sigma\lambda+a\rho/2},\\
	\hat\tau_0(p\otimes \vac_{\lambda+a\rho_2})&:=(-1)^{k_1k_2}(\tau\cdot p)\otimes\vac_{\tau\lambda+a\rho/2},
	\end{split}
\end{gather}
for $p\in S(\hat{\gt{h}}^-)$ and $\lambda=k_1\le_1+k_2\le_2+k_3\le_3$.

Now consider $V_{D(L)}=(V_{D(L)},\widetilde{Y},\vac_0,\omega)$, where $D$ is the cone determined by the basis $\varepsilon_i$,
\begin{gather}
	D=\left\{\sum_{i=1}^3 \alpha_i\varepsilon_i\in L\otimes_{\ZZ}\QQ \mid \alpha_i\geq 0,\,\forall i, \text{ or } \alpha_i<0\,,\forall i\right\},
\end{gather}
and $\widetilde{Y}$ is the vertex operator correspondence defined by (\ref{eqn:va:cva-vdvops}) in \S\ref{sec:va:cva}. Observe that if we set 
\begin{gather}\label{eqn:va:cnstn-leprime}
\le_i':=2\rho-\le_i
\end{gather}
for $i\in\{1,2,3\}$ then $\lab \le_i',\le_j\rab=\delta_{i,j}$. Since the values $\lab \varepsilon_i,\varepsilon_j\rab$ are all positive, there are no non-trivial vectors $\lambda\in D$ with $\lab \lambda,\lambda\rab\leq 0$. So, by virtue of Theorem \ref{thm:va:cva-VD}, the super vertex algebra $V_{D(L)}$ becomes a super vertex operator algebra, with central charge $c=3$, when equipped with the conformal element
\begin{gather}
	\omega=\frac{1}{2}\sum_{i=1}^3\le_i'(-1)\le_i(-1)\otimes \vac_0. 
\end{gather}
Observe that the actions (\ref{eqn:va:cnstn-sigma}) and (\ref{eqn:va:cnstn-tau}), of $\hat\sigma_0$ and $\hat\tau_0$, respectively, restrict from $V_L$ to $V_{D(L)}$, since $D$ is invariant under coordinate permutations. We define automorphisms $\hat\sigma$ and $\hat\tau$ for $V_{D(L)}$, by taking $\hat\sigma:=\hat\sigma_0$ and $\hat\tau:=\hat\tau_0\circ s$, where $s$ is the sign automorphism of $V_{D(L)}$, defined in \S\ref{sec:va:cva}. Since $s$ has order two and commutes with $\hat\tau_0$ we see that $\hat\sigma$ and $\hat\tau$ generate a copy of $S_3$ in $\Aut(V_{D(L)})$, and we denote this group $\hat{G}$.
\begin{gather}
\hat{G}:=\lab \hat\sigma,\hat\tau\rab<\Aut(V_{D(L)})
\end{gather}

Theorem \ref{thm:va:cva-VDh} and the discussion above furnish us with canonically-twisted $V_{D(L)}$-modules $V_{D(L+a\rho/2)}$ for $a$ an odd integer. Note that this furnishes five distinct canonically-twisted $V_{D(L)}$-modules, since the isomorphism type of $V_{D(L+a\rho/2)}$ is determined by $a\pmod{10}$, since $k=10$ is the minimal positive integer such that $k\rho/2\in L$. We extend the action of the canonical involution $g_{\rho/2}$ 
from $V_{D(L)}$ to $V_{D(L+a\rho/2)}$ just as we do for $V_L$-modules (cf. (\ref{eqn:va:latvamod-sigmahtw})), 
by setting 
\begin{gather}\label{eqn:sigmaVxtn}
g_{\rho/2}(p\otimes \vac_{\lambda+a\rho/2}):=(-1)^{\lab \rho,\lambda\rab}p\otimes \vac_{\lambda+a\rho/2}
\end{gather}
for $p\in S(\hat{\gt{h}}^-)$ and $\lambda+a\rho/2\in D(L+a\rho/2)$. 
Similarly, we extend the actions of $\hat\sigma$ and $\hat\tau$, from $V_{D(L)}$ to $V_{D(L+a\rho/2)}$,
\begin{gather}
	\begin{split}\label{eqn:va:cnstn-sigtauVDLtw}
	\hat\sigma(p\otimes \vac_{\lambda+a\rho/2})&:=(-1)^{k_2k_3+k_3k_1}(\sigma\cdot p)\otimes\vac_{\sigma\lambda+a\rho/2},\\
	\hat\tau(p\otimes \vac_{\lambda+a\rho_2})&:=
	\begin{cases}
	(-1)^{k_1k_2}(\tau\cdot p)\otimes\vac_{\tau\lambda+a\rho/2},&\text{ if $\lambda+a\rho/2\in P$,}\\
	(-1)^{k_1k_2+1}(\tau\cdot p)\otimes\vac_{\tau\lambda+a\rho/2},&\text{ if $\lambda+a\rho/2\in N$,}
	\end{cases}
	\end{split}
\end{gather}
and thus obtain actions of $\hat{G}$ on the canonically-twisted $V_{D(L)}$-modules, $V_{D(L+a\rho/2)}$.
In (\ref{eqn:va:cnstn-sigtauVDLtw}) we write $p$ for an element of $S(\hat{\gt{h}}^-)$, and assume $\lambda=k_1\le_1+k_2\le_2+k_3\le_3$.

We now let $V^X$ denote the tensor product super vertex operator algebra 
\begin{gather}\label{eqn:va:cnstn-VX}
V^X:=A(\gt{p})\otimes V_{D(L)}.
\end{gather} 
We write $V_{\tw,a}^\pm$ for the canonically-twisted $V^X$-module, 
\begin{gather}\label{eqn:va:cnstn-Vtw}
V^{\pm}_{\tw,a}:=A(\gt{p})_\tw^{\pm}\otimes V_{D(L+a\rho/2)}.
\end{gather} 
We extend the action of $\hat{G}\simeq S_3$ from $V_{D(L)}$ to $V^X$, and from $V_{D(L+a\rho/2)}$ to $V^{\pm}_{\tw,a}$, by letting $\hat{G}$ act trivially on the Clifford module factors, setting
\begin{gather}
	\hat\sigma(u\otimes v):=u\otimes \hat\sigma(v),\quad
	\hat\tau(u\otimes v):=u\otimes \hat\tau(v),
\end{gather}
for $u\in A(\gt{p})$ and $v\in V_{D(L)}$, and for $u\in A(\gt{p})_\tw^\pm$ and $v\in V_{D(L+a\rho/2)}$.

Given $g\in \hat{G}$ and $a$ an odd integer, we now define $T^{\pm}_{g,a}$ to be the trace of the operator $gg_{\rho/2}p(0)q^{L(0)-c/24}$ on the canonically-twisted $V^X$-module $V^{\pm}_{\tw,a}$,
\begin{gather}\label{eqn:va:cnstn-Tpmga}
	T^{\pm}_{g,a}:=
	\tr_{V^{\pm}_{\tw,a}}gg_{\rho/2}p(0)q^{L(0)-c/24}.
\end{gather}

Recall that $(q;q)_\infty=\prod_{n>0}(1-q^n)$ (cf. (\ref{eqn:intro-poch})). Our concrete construction allows us to compute explicit formulas for the trace functions $T^{\pm}_{g,a}$.

\begin{prop}\label{prop:va:cnstn-tracefnexpressions}
The trace functions $T^{\pm}_{g,a}$ admit the following expressions, for $a\in\{1,3,5,7,9\}$.
\begin{align}
	T^{\pm}_{e,a}&=\pm\frac{q^{-1/12}}
		{(q;q)^2_{\infty}}
	\left(\sum_{k,l,m\geq 0}+\sum_{k,l,m<0}\right)
	\label{eqn:va:cnstn-Tpmeaexplicit}
	(-1)^{k+l+m}q^{(k^2+l^2+m^2)/2+2(kl+lm+mk)+a(k+l+m)/2+3a^2/40}\\
	T^{\pm}_{\hat\tau,a}&=\pm\frac{q^{-1/12}}
		{(q^2;q^2)_{\infty}}
	\left(\sum_{k,m\geq 0}-\sum_{k,m<0}\right)\label{eqn:va:cnstn-Tpmtauaexplicit}
	(-1)^{k+m}q^{3k^2+m^2/2+4km+a(2k+m)/2+3a^2/40}\\
	T^\pm_{\hat\sigma,a}&=\pm q^{-1/12}\frac{(q;q)_{\infty}}{(q^3;q^3)_{\infty}}\label{eqn:va:cnstn-Tpmsigmaaexplicit}
	\sum_{k\in\ZZ}(-1)^kq^{15k^2/2+3ak/2+3a^2/40}
\end{align}
\end{prop}

\begin{proof}
First consider the case that $g=e$ is the identity. From the definition (\ref{eqn:va:cnstn-Tpmga}) of $T^{\pm}_{e,a}$ we derive 
\begin{gather}\label{eqn:va:cnstn-Tpmeadirect}
	T^{\pm}_{e,a}
	=\pm\frac{1}
	{(q;q)_{\infty}^2}
	\sum_{\mu\in D(L+a\rho/2)}(-1)^{\lab \mu-a\rho/2,\rho\rab}q^{\lab \mu,\mu\rab/2-1/12},
\end{gather}
for any odd integer $a$. If also $0<a<10$ then $D(L+a\rho/2)=D(L)+a\rho/2$, and so in this situation we may replace $\mu$ with $k\le_1+l\le_2+m\le_3+a\rho/2$ in the summation, where either $k,l,m\geq 0$ or $k,l,m<0$. This leads to (\ref{eqn:va:cnstn-Tpmeaexplicit}) directly, according to the definition (\ref{eqn:va:cnstn-bilinearform}) of $\lab\cdot\,,\cdot\rab$, and the identity (\ref{eqn:va:cnstn-iplambdarho}). The term $3a^2/40$ appears because $\lab\rho,\rho\rab=3/5$.

Next take $g=\hat\tau$. We compute 
\begin{gather}\label{eqn:va:cnstn-Tpmtauadirect}
	T^{\pm}_{\hat\tau,a}(q)=\pm\frac{1}{(q^2;q^2)_{\infty}}
	\left(\sum_{\substack{\mu\in P(L+a\rho/2)\\ \tau\mu=\mu}}-
	\sum_{\substack{\mu\in N(L+a\rho/2)\\ \tau\mu=\mu}}\right)
	(-1)^{\lab \mu-a\rho/2,\rho+\le_1'\rab}q^{\lab \mu,\mu\rab/2-1/12}
\end{gather}
using the definition (\ref{eqn:va:cnstn-Tpmga}) of $T^{\pm}_{g,a}$, and the formula (\ref{eqn:va:cnstn-sigtauVDLtw}) for the action of $\hat\tau$. (See also (\ref{eqn:va:cnstn-leprime}).) Note that the sign change for summands with $\mu\in N(L+a\rho_2)$ is a consequence of the fact that the action of $\hat\tau$ is defined by composing $\hat\tau_0$ (cf. (\ref{eqn:va:cnstn-sigtaunoughtVLtw})) with the sign automorphism $s$ (cf. (\ref{eqn:va:cva-signaut})).
Restricting to $0<a<10$, we obtain (\ref{eqn:va:cnstn-Tpmtauaexplicit}) from (\ref{eqn:va:cnstn-Tpmtauadirect}) in much the same way as above, by taking $\mu=k\le_1+k\le_2+m\le_3+a\rho/2$ in the summations, with $k,m\geq 0$ in the first of these, and $k,m<0$ in the second. The factor $(-1)^k$ in $(-1)^{k+m}$, corresponding to $(-1)^{\lab \mu-a\rho/2,\le_1'\rab}$ in (\ref{eqn:va:cnstn-Tpmtauadirect}), arises from the factor $(-1)^{k_1k_2}=(-1)^{k^2}=(-1)^k$ in (\ref{eqn:va:cnstn-sigtauVDLtw}). 

Finally we consider $g=\hat\sigma$ (cf. (\ref{eqn:va:cnstn-sigtauVDLtw})). Then the appropriate analogue of (\ref{eqn:va:cnstn-Tpmeadirect}) and (\ref{eqn:va:cnstn-Tpmtauadirect}) is
\begin{gather}\label{eqn:va:cnstn-Tpmsigmaadirect}
	T^{\pm}_{\hat\sigma,a}(q)=\pm\frac{(q;q)_{\infty}}{(q^3;q^3)_{\infty}}
	\sum_{\substack{\mu\in D(L+a\rho/2)\\ \sigma\mu=\mu}}
	(-1)^{\lab \mu-a\rho/2,\rho\rab}q^{\lab \mu,\mu\rab/2-1/12}.
\end{gather}
We obtain (\ref{eqn:va:cnstn-Tpmsigmaaexplicit}) from (\ref{eqn:va:cnstn-Tpmsigmaadirect}), by restricting to $0<a<10$, and substituting $\mu=k\le_1+k\le_2+k\le_3+a\rho/2=(5k+a/2)\rho$ in the summation. This completes the proof of the proposition.
\end{proof}

\section{Mock Theta Functions}\label{sec:mcktht}

In this section we consider the modular properties of the trace functions defined in \S\ref{sec:va:cnstn}, computed explicitly in Proposition \ref{prop:va:cnstn-tracefnexpressions}. We recall some basic facts about Maass forms in \S\ref{sec:mcktht:maass}, including their relationship to mock modular forms. 
We require some facts about theta series of cones in indefinite lattices due to Zwegers \cite{Zwegers}, which we recall in \S\ref{sec:mcktht:indtht}. The proof of our main result, Theorem \ref{thm:intro-maintheorem}, appears in \S\ref{sec:mcktht:um}. In particular, we identify the umbral McKay--Thompson series attached to $X=E_8^3$ as trace functions arising from the action of $G^X$ on canonically-twisted modules for $V^X$ in \S\ref{sec:mcktht:um}.

\subsection{Harmonic Maass Forms}\label{sec:mcktht:maass}

Define the weight $1/2$ {\em Casimir operator} $\Omega_{\tfrac12}$, a differential operator on smooth functions $H:\HH\to\CC$, by setting
\begin{gather}\label{eqn:mcktht:um-cas}
	(\Omega_{\frac12}H)(\tau):=
	-4\Im(\tau)^2\frac{\partial^2H}{\partial\tau\partial\overline{\tau}}(\tau)
	+i\Im(\tau)\frac{\partial H}{\partial\overline{\tau}}(\tau)
	+\frac{3}{16}H(\t).
\end{gather} 
Note that $\Omega_{\tfrac12}=\Delta_{\tfrac 12}+\tfrac{3}{16}$, where $\Delta_{k}$ is the hyperbolic Laplace operator in weight $k$. 

Following the work \cite{BruFun} of Bruinier--Funke (cf. \cite{ono_unearthing,zagier_mock}), a {\em harmonic weak Maass form} of weight $1/2$ for $\Gamma<\SL_2(\ZZ)$ is defined to be a smooth function $H:\HH\to \CC$ that transforms as a (not necessarily holomorphic) modular form of weight $1/2$ for $\Gamma$, is an eigenfunction for $\Omega_{\frac12}$ with eigenvalue $3/16$, and has at most exponential growth as $\tau$ approaches cusps of $\Gamma$. 

Define $\beta(x)$ for $ x \in \RR_{ \ge 0}$ by setting
\begin{gather}\label{eqn:mcktht:indtht-beta}
\beta(x) := \int_x^\infty u^{-1/2} e^{- \pi u} {\rm d}u.
\end{gather}
Note that $\beta$ is related to the incomplete Gamma function by $\sqrt{\pi}\beta(x)=\Gamma(1/2,\pi x)$. If $H$ is a harmonic weak Maass form of weight $1/2$ then we can canonically decompose $H$ into its {\em holomorphic} and {\em non-holomorphic} parts, $H=H^++H^-$, where
\begin{align}
	H^+(\tau)&
	=\sum_{n\gg -\infty}c_H^+(n)q^n,\label{eqn:mcktht:um-Hp}\\
	H^-(\tau)&
	=2ic_H^-(0)\sqrt{2\Im(\tau)}-i\sum_{n>0}c_H^-(n){\frac{1}{\sqrt{2n}}}\beta(4n\Im(\tau))q^{-n},\label{eqn:mcktht:um-Hm}
\end{align}
for some uniquely determined values $c_H^{\pm}(n)\in \CC$. (Cf.  \S3 of \cite{BruFun}. See also \S5 of \cite{zagier_mock} and \S7.1 of \cite{Dabholkar:2012nd}.) Note that $n$ should be allowed to range over rational values in (\ref{eqn:mcktht:um-Hp}) and (\ref{eqn:mcktht:um-Hm}).

We may define the {\em mock modular forms} of weight $1/2$ to be those holomorphic functions $H^+:\HH\to\CC$ which arise as the holomorphic parts of harmonic weak Maass forms of weight $1/2$. For $H^\pm$ as above, the {\em shadow} of $H^+$ is defined, up to a choice of scaling factor $C$, by
\begin{gather}\label{eqn:mcktht:um-Hpshadow}
	g(\tau):=C{\sqrt{2\Im(\tau)}}\overline{\frac{\partial H^-}{\partial\overline{\tau}}}=C\sum_{n\geq 0}c^-_H(n)q^n.\end{gather}
Then so long as $c_H^-(0)=0$ (i.e. $g$ is a cusp form), the function $H^-$ is the {\em Eichler integral} of $g$, 
\begin{gather}\label{eqn:mcktht:um-Hmshadow}
	H^-(\tau)=
	\frac{e(-\tfrac18)}{C}\int_{-\overline{\tau}}^\infty \frac{\overline{g(-\overline{z})}}{\sqrt{z+\tau}}{\rm d} z.
\end{gather}
In this setting, the weak harmonic Maass form $H=H^++H^-$ is called the {\em completion} of $H^+$.

Various choices for $C$ can be found in the literature. In \cite{MUM} we find $C=\sqrt{2m}$ in the case that $H=(H_r)$ is a $2m$-vector-valued Maass form for some $\Gamma_0(N)$, 
such 
that 
\begin{gather}
(H\cdot\theta)(\t,z):=\sum_r H_r(\tau)\theta_{m,r}(\tau,z)
\end{gather}
transforms likes a (not necessarily holomorphic in $\tau$) Jacobi form of weight $1$ and index $m$ for $\Gamma_0(N)$, where
\begin{gather}\label{eqn:mcktht:maass-tht}
	\theta_{m,r}(\tau,z):=\sum_{k\in\ZZ}q^{(2km+r)^2/4m}e^{2\pi i z(2km+r) }.
\end{gather}

The cases of relevance to us here all have $m=30$, so we take $C=\sqrt{60}$ henceforth in (\ref{eqn:mcktht:um-Hpshadow}) and (\ref{eqn:mcktht:um-Hmshadow}). All the shadows arising in this work will be linear combinations of the unary theta functions 
\begin{gather}\label{eqn:mcktht-Smr}
S_{m,r}(\tau):=\left.\frac{1}{2\pi i}\frac{\partial}{\partial z}\theta_{m,r}(\tau,z)\right|_{z=0}=\sum_{k\in\ZZ}(2km+r)q^{(2km+r)^2/4m},
\end{gather}
where $m=30$ and $r\neq 0 \pmod{30}$.
In particular, we will not encounter any examples for which the shadow $g$ (cf. (\ref{eqn:mcktht:um-Hpshadow})) is not a cusp form.

\subsection{Indefinite Theta Series}\label{sec:mcktht:indtht}

We will be concerned with quadratic forms of signature $(1,1)$, and so take $r=2$ in the notation of \cite{Zwegers}. (Even though our main construction uses a lattice of signature $(1,2)$, it will develop in \S\ref{sec:mcktht:um} that the trace functions (\ref{eqn:va:cnstn-Tpmeaexplicit}) and (\ref{eqn:va:cnstn-Tpmtauaexplicit}) can be analyzed in terms of theta series of indefinite lattices with signature $(1,1)$. The remaining trace function (\ref{eqn:va:cnstn-Tpmsigmaaexplicit}) is essentially a theta series with rank $1$, and consequently can be handled by classical methods.)

Given a symmetric $2\times 2$ matrix $A$, we define a quadratic form $Q: \RR^2 \to \RR$, by setting
\begin{equation}
Q(x):= \frac{1}{2} ( x,A x),
\end{equation}
where $(\cdot\,,\cdot)$ denotes the usual Euclidean inner product on $\RR^2$. The associated bilinear form is
\begin{equation}
B(x,y):= ( x, A y) = Q(x+y)-Q(x)-Q(y) \, .
\end{equation}
Henceforth assume that $A$ has signature $(1,1)$. Then the set of vectors $c \in \RR^2$ with $Q(c)<0$ is non-empty and has two components.
Let $C_Q$ be one of these components. Two vectors $c^{(1)},c^{(2)}$ belong to the same component if $B(c^{(1)},c^{(2)})<0$. Thus, picking a vector $c_{0}$
in $C_Q$ we may identify
\begin{equation}
C_Q= \left\{ c \in \RR^2 \mid Q(c)<0, ~B(c,c_0)<0 \right\} \, .
\end{equation}
Zwegers also defines a set of representatives of {\em cusps},
\begin{equation}
S_Q:= \left\{ c \in \ZZ^2 \mid \text{$c$ primitive, $Q(c)=0$, $B(c,c_0)<0$} \right\} \, .
\end{equation}

Define the {\em indefinite theta function} with characteristics $a, b \in \RR^2$, with respect to
$c^{(1)}, c^{(2)} \in C_Q$, by setting
\begin{gather}
\begin{split}\label{eqn:mcktht:indtht-vartheta}
&\vartheta^{c^{(1)},c^{(2)}}_{a,b}(\tau) :=\\
& \sum_{\nu \in a + \ZZ^2}  \left( E \left( \frac{B(c^{(1)},\nu)}{\sqrt{-Q(c^{(1)})}} \sqrt{\Im(\tau)} \right)  -E \left( \frac{B(c^{(2)},\nu)}{\sqrt{-Q(c^{(2)})}} \sqrt{\Im(\tau)} \right) \right) 
q^{Q(\nu)} e^{2 \pi i B(\nu,b)},
\end{split}
\end{gather}
where $E(z) := \sgn (z) (1-\beta(z^2))$. 
Corollary 2.9 of \cite{Zwegers} (cf. also Theorem 3.1 of \cite{zagier_mock}) shows that $\vartheta^{c^{(1)},c^{(2)}}_{a,b}(\tau)$ is a non-holomorphic modular form of weight $1$. 

Presently we will see that these indefinite theta functions can be used to define harmonic Maass forms whose non-holomorphic parts can be written in terms of the 
functions
\begin{equation}\label{eqn:mcktht-Rab}
R_{a,b}(\tau) := \sum_{\nu \in a+\ZZ} \sgn (\nu) \beta(2 \nu^2 \Im(\tau)) q^{-\nu^2/2} e^{- 2 \pi i \nu b}.
\end{equation}

Note that the $R_{a,b}$ are Eichler integrals (cf. (\ref{eqn:mcktht:um-Hmshadow})) of unary theta functions of weight $3/2$. Indeed, we have
\begin{equation}\label{eqn:mcktht-Rabgab}
R_{a,b}(\tau) 
=e(-\tfrac18) \int_{- \bar \tau}^{i \infty} \frac{g_{a,-b}(z)}{\sqrt{z+\tau}}{\rm d}z,
\end{equation}
for 
\begin{equation}\label{eqn:mcktht-gab}
g_{a,b}(\tau) :=\sum_{\nu \in a+\ZZ} \nu q^{\nu^2/2} e^{2 \pi i \nu b}.
\end{equation}
Observe also that
\begin{equation}\label{eqn:mcktht-gabSmr}
g_{\frac{r}{2m},0}(m \tau) = \frac{1}{2m} S_{m,r}(\tau)
\end{equation}
(cf. (\ref{eqn:mcktht-Smr})), which is useful for comparing the results of \cite{Zwegers} to those of \cite{MUM}.

Define $\langle c \rangle_\ZZ^\perp:=\{ \xi \in \ZZ^r \mid B(c, \xi)=0 \}$. For future use we quote the $r=2$ case of Proposition 4.3 from \cite{Zwegers}.
\begin{prop}[Zwegers]\label{prop:mcktht-Zwegersprop}
Let $c \in C_Q \cap \ZZ^2$ be primitive. Let $P_0\subset \RR^2$ be the finite set determined by requiring that
\begin{gather}
\left\{ \mu \in a+\ZZ^2  \mid 0 \leq \frac{B(c,\mu)}{2 Q(c)} <1 \right\} 
=  \bigsqcup_{\mu_0 \in P_0} \left( \mu_0 + \langle c \rangle_\ZZ^\perp \right).
\end{gather}
Then we have
\begin{gather}
	\begin{split}\label{eqn:mcktht-Zwegersprop}
\sum_{\nu \in a + \ZZ^2} &\sgn \left( B(c,\nu) \right) \beta \left( - \frac{B(c,\nu)^2}{Q(c)} \Im(\tau) \right) 
e^{2 \pi i Q(\nu) \tau + 2 \pi i B(\nu,b)} \\
& = - \sum_{\mu_0 \in P_0} R_{\frac{B(c,\mu_0)}{2Q(c)},B(c,b)} (-2 Q(c) \tau) \cdot \sum_{\xi \in \mu_0^\perp + \langle c \rangle_\ZZ^\perp} e^{2 \pi i Q(\xi) \tau + 2 \pi i B(\xi,b^\perp)},
	\end{split}
\end{gather}
where $\mu_0^\perp= \mu_0 - \frac{B(c,\mu_0)}{2 Q(c)} c $ and $b^\perp= b - \frac{B(c,b)}{2 Q(c)} c$.
\end{prop}
Note that the term 
\begin{equation}
 \sum_{\xi \in \mu_0^\perp + \langle c \rangle_\ZZ^\perp} e^{2 \pi i Q(\xi) \tau + 2 \pi i B(\xi,b^\perp)}
\end{equation}
is a classical (positive-definite) theta function of weight $1/2$.

The indefinite theta function construction (\ref{eqn:mcktht:indtht-vartheta}) is applied to mock theta functions of Ramanujan (other than $\chi_0$ and $\chi_1$, which are treated in \cite{MR2558702}) in \cite{Zwegers}. 
Amongst those appearing are the four functions $F_0$, $F_1$, $\phi_0$ and $\phi_1$, where $\phi_0$ and $\phi_1$ are defined in (\ref{eqn:intro-phi01}), and
\begin{gather}
	\begin{split}\label{eqn:mcktht:um-F01}
	F_0(q)&:=\sum_{n\geq 0}\frac{q^{2n^2}}{(q;q^2)_n},\\
	F_1(q)&:=\sum_{n\geq 0}\frac{q^{2n(n+1)}}{(q;q^2)_{n+1}}.
	\end{split}
\end{gather}
These are amongst the fifth order mock theta functions introduced by Ramanujan in his last letter to Hardy. 

To study these functions Zwegers introduces $6$-vector-valued mock modular forms 
\begin{gather}
	F_{5,1}(\tau)=(F_{5,1,r}(\tau)),\quad F_{5,2}(\tau)=(F_{5,2,r}(\tau)),
\end{gather} 
on
pages 74 and 79, respectively, of \cite{Zwegers}. Inspecting their definitions, and substituting $2\tau$ for $\tau$, 
we find that
\begin{align}	
	F_{5,1,3}(2\tau)&=q^{-1/120}(F_0(q)-1), &F_{5,2,3}(2\t)&=q^{-1/120}\phi_0(-q),\label{eqn:mcktht:indtht-F5Fphi0}\\
	F_{5,1,4}(2\tau)&=q^{71/120}F_1(q), &F_{5,2,4}(2\t)&=-q^{-49/120}\phi_1(-q).\label{eqn:mcktht:indtht-F5Fphi1}
\end{align}

The content of Proposition 4.10 of \cite{Zwegers} is that
\begin{gather}\label{eqn:mcktht:um-HFG1}
	H_{5,1}(\tau)=F_{5,1}(\tau)-G_{5,1}(\tau),
\end{gather}
where the vector-valued functions $H_{5,1}$ and $G_{5,1}$ are such that the components of $2\eta(\tau) H_{5,1}(\tau)$ are non-holomorphic indefinite theta functions of the form $\vartheta_{a,b}^{c^{(1)},c^{(2)}}(\tau)$ (cf. (\ref{eqn:mcktht:indtht-vartheta})), and the third and fourth components of $G_{5,1}$ satisfy
\begin{gather}\label{eqn:mcktht:um-G3RRRR}
	G_{5,1,3}(2\tau)=-\frac{1}{2}\left(R_{\frac{19}{60},0}+R_{\frac{29}{60},0}-R_{\frac{49}{60},0}-R_{\frac{59}{60},0}\right)(60\tau), \\ 
	G_{5,1,4}(2\tau)=-\frac{1}{2}\left(R_{\frac{13}{60},0}+R_{\frac{23}{60},0}-R_{\frac{43}{60},0}-R_{\frac{53}{60},0}\right)(60\tau).
	\label{eqn:mcktht:um-G4RRRR}
\end{gather}
(Cf. (\ref{eqn:mcktht-Rab}) for $R_{a,b}$.) Moreover, $H_{5,1}(\tau)$ is an eigenfunction for $\Omega_{\frac12}$ 
with eigenvalue $3/16$ (cf. (\ref{eqn:mcktht:um-cas})). In other words, the components of 
$H_{5,1}=(H_{5,1,r})$ are 
harmonic weak Maass forms of weight $1/2$ (cf. \S\ref{sec:mcktht:maass}).

Proposition 4.13 of \cite{Zwegers} establishes a similar result for $F_{5,2}$, namely
\begin{gather}\label{eqn:mcktht:um-HFG2}
	H_{5,2}(\tau)=F_{5,2}(\tau)-G_{5,2}(\tau),
\end{gather}
where $H_{5,2}$ is again a harmonic weak Maass form of weight $1/2$, 
and $G_{5,2}=-G_{5,1}$. 

The left hand sides of (\ref{eqn:mcktht:um-HFG1}) and (\ref{eqn:mcktht:um-HFG2}) are harmonic weak Maass forms of weight $1/2$, so they admit canonical decompositions into holomorphic (cf. (\ref{eqn:mcktht:um-Hp})) and non-holomorphic (cf. (\ref{eqn:mcktht:um-Hm})) parts. The summands $F_{5,1}$ and $F_{5,2}$ on the right hand sides are holomorphic by construction, and the $R_{a,b}$ are of the same form as (\ref{eqn:mcktht:um-Hm}) by construction (cf. (\ref{eqn:mcktht-Rab})), so the right hand sides of (\ref{eqn:mcktht:um-HFG1}) and (\ref{eqn:mcktht:um-HFG2}) are precisely the decompositions of $H_{5,1}$ and $H_{5,2}$ into its holomorphic and non-holomorphic parts. 

Equivalently, the four functions $F_{5,j,r}$ are mock modular forms of weight $1/2$ with completions given by the $H_{5,j,r}$, and the $G_{5,j,r}$ are the Eichler integrals of their shadows. Thus we can describe their shadows explicitly. Applying (\ref{eqn:mcktht-Rabgab}), (\ref{eqn:mcktht-gab}) and (\ref{eqn:mcktht-gabSmr}), and the identities $g_{1-a,0}=g_{-a,0}=-g_{a,0}$, we see that $F_{5,1,3}(2\tau)$ and $-F_{5,2,3}(2\tau)$
have the same shadow
\begin{gather}\label{eqn:mcktht:indtht-F5123shadow}
\frac12(S_{30,1}+S_{30,11}+S_{30,19}+S_{30,29})(\tau),
\end{gather}
while $F_{5,1,4}(2\tau)$ and $-F_{5,2,4}(2\tau)$ both have shadow given by
\begin{gather}\label{eqn:mcktht:indtht-F5124shadow}
\frac12(S_{30,7}+S_{30,13}+S_{30,17}+S_{30,27})(\tau).
\end{gather}

\subsection{McKay--Thompson Series}\label{sec:mcktht:um}

In this section we prove our main result, 
Theorem \ref{thm:intro-maintheorem}, that the trace functions 
arising from the action of $G^X$ on the $V^\pm_{\tw,a}$ 
recover the Fourier expansions of the mock modular forms $H^X_g$ attached to $g\in G^X\simeq S_3$ by umbral moonshine at $X=E_8^3$.

To formulate this precisely, let $T^X_{g}=(T^X_{g,r})$ be the vector of Laurent series in (rational powers of) $q$, with components indexed by $\ZZ/60\ZZ$, such that
\begin{gather}\label{eqn:mcktht-TXg}
	T^X_{g,r}:=\begin{cases}
	T^{\mp}_{g,1},&\text{ for $r=\pm 1,\pm 11,\pm 19,\pm 29\pmod{60}$,}\\
	T^{\mp}_{g,7},&\text{ for $r=\pm 7,\pm 13,\pm 17,\pm 23\pmod{60}$,}\\
	0,&\text{ else,}
	\end{cases}
\end{gather}
and define the {\em polar part at infinity} of $T^X_{g}$ to be the vector of polynomials in (rational powers of) $q^{-1}$ obtained by removing all non-negative powers of $q$ in each component $T^X_{g,r}$. Let $g\mapsto\bar{\chi}_g^X$ be the natural permutation character of ${G^X}$, so that $\bar{\chi}_g$ is $3$, $1$ or $0$, according as $g$ has order $1$, $2$ or $3$, and define a vector $S^X_g=(S^X_{g,r})$ of theta series, with components indexed by $\ZZ/60\ZZ$, by setting
\begin{gather}\label{eqn:mcktht-SXg}
	S^X_{g,r}:=\begin{cases}
			\pm\bar{\chi}_g(S_{30,1}+S_{30,11}+S_{30,19}+S_{30,29}),&\text{ if $r=\pm1,\pm11,\pm19,\pm 29\pmod{60}$,}\\
			\pm\bar{\chi}_g(S_{30,7}+S_{30,13}+S_{30,17}+S_{30,23}),&\text{ if $r=\pm7,\pm13,\pm17,\pm 23\pmod{60}$,}\\
			0&\text{ else.}
        		\end{cases}
\end{gather}
(Cf. (\ref{eqn:mcktht-Smr}).) 

Set $S^X:=S^X_e$, and let $\sigma^X:\SL_2(\ZZ)\to\GL_{60}(\CC)$ denote the multiplier system of $S^X$, so that
\begin{gather}\label{eqn:mcktht:um-sigmaX}
	\sigma^X(\gamma)S^X(\gamma\tau)(c\tau+d)^{-3/2}=S^X(\tau)		
\end{gather}
for $\tau \in\HH$ and $\gamma\in \SL_2(\ZZ)$, when $(c,d)$ is the lower row of $\gamma$.
Our next goal (to be realized in Proposition \ref{prop:mcktht:um-TXg}) is to show that $2T^X_g$ is a mock modular form with shadow $S^X_g$ for $g\in G^X$. This condition tells us what the multiplier system of $T^X_g$ must be, at least when $o(g)$ is $1$ or $2$ (as $S^X_g$ is identically zero when $o(g)=3$). 
For the convenience of the reader we describe this multiplier system in more detail now. 

It is cumbersome to work with matrices in $\GL_{60}(\CC)$, but we can avoid this since any non-zero component of $T^X_g$ is $\pm1$ times $T^X_{g,1}$ or $T^X_{g,7}$. That is, we can work with the $2$-vector-valued functions $\check T^X_g:=(T^X_{g,1},T^X_{g,7})$ and $\check S^X_{g}:=(S^X_{g,1},S^X_{g,7})$. If $h=(h_r)$ is a modular form of weight $1/2$ with multiplier system conjugate to that of $S^X$, and satisfying 
\begin{gather}\label{eqn:mcktht-hr}
	h_{r}:=\begin{cases}
	h_{1},&\text{ for $r=\pm 1,\pm 11,\pm 19,\pm 29\pmod{60}$,}\\
	h_{7},&\text{ for $r=\pm 7,\pm 13,\pm 17,\pm 23\pmod{60}$,}\\
	0,&\text{ else,}
	\end{cases}
\end{gather}
then, setting $\check h=(h_1,h_7)$, we have
\begin{gather}
	\check{h}\left(\frac{a\t+b}{c\t+d}\right)\check\nu\left(\frac{a\t+b}{c\t+d}\right)
	(c\tau+d)^{-1/2}=\check{h}(\t)
\end{gather}
where $\check\nu:\SL_2(\ZZ)\to\GL_2(\CC)$ is determined  by the rules
\begin{gather}
	\begin{split}\label{eqn:mcktht-checknu}
	\check\nu
	\begin{pmatrix}
		1&1\\
		0&1
	\end{pmatrix}
	&=
	\begin{pmatrix}
		e(-\tfrac{1}{120})&0\\
		0&e(-\tfrac{49}{120})
	\end{pmatrix},\\
	\check\nu
	\begin{pmatrix}
		0&-1\\
		1&0
	\end{pmatrix}
	&=\frac{2e(\frac{3}{8})}{\sqrt{15}}
	\begin{pmatrix}
		\sin(\pi\tfrac{1}{30})+\sin(\pi\tfrac{11}{30})&\sin(\pi\frac{7}{30})+\sin(\pi\frac{13}{30})\\
		\sin(\pi\tfrac{7}{30})+\sin(\pi\frac{13}{30})&-\sin(\pi\frac{1}{30})-\sin(\pi\frac{11}{30})
	\end{pmatrix}.
	\end{split}
\end{gather}

We now return to our main objective: the determination of the modularity of $T^X_g$ for $g\in {G^X}$.
To describe the multiplier system for $T^X_g$ when $o(g)=3$ we require the function $\rho_{3|3}:\Gamma_0(3)\to \CC^\times$, defined by setting 
\begin{gather}\label{eqn:mcktht:um-rho33}
	\rho_{3|3}\left(\begin{matrix}a&b\\c&d\end{matrix}\right):=e\left(\frac{cd}{9}\right).
\end{gather}
Evidently $\rho_{3|3}$ has order $3$, and restricts to the identity on $\Gamma_0(9)$.

\begin{prop}\label{prop:mcktht:um-TXg}
Let $g\in G^X$. Then $2T^X_g$ is the Fourier series of a mock modular form for $\Gamma_0(o(g))$ whose shadow is $S^X_g$. The polar part at infinity of $2T^X_g$ is given by 
\begin{gather}
	T^X_{g,r}=\begin{cases} \mp 2q^{-1/120}+O(1),&\text{ if $r=\pm 1,\pm11,\pm19,\pm29\pmod{60}$,}\\
	O(1),&\text{ otherwise,}
	\end{cases}
\end{gather}
and 
$2T^X_g$ has vanishing polar part at all non-infinite cusps of $\Gamma_0(o(g))$. If $o(g)=3$ then the multiplier system of $2T^X_g$ is given by $\gamma\mapsto \rho_{3|3}(\gamma)\overline{\sigma^X(\gamma)}$. 
\end{prop}
\begin{proof}

According to our definition (\ref{eqn:mcktht-TXg}), the components of $T^X_g$ are $T^{\pm}_{g,1}$ or $T^{\pm}_{g,7}$. In practice it is more convenient to work with $T^{\pm}_{g,3}$ than $T^{\pm}_{g,7}$, and we may do so because these functions coincide up to a sign (depending upon $g$). To see this, observe that $D(L+a\rho/2)=-D(L-a\rho/2)$ for $a$ an odd integer. Then comparing with the expressions (\ref{eqn:va:cnstn-Tpmeadirect}), (\ref{eqn:va:cnstn-Tpmtauadirect}) and (\ref{eqn:va:cnstn-Tpmsigmaadirect}), we see that $T^{\pm}_{g,a}=T^{\pm}_{g,-a}$ when $o(g)=1$ or $3$, and $T^{\pm}_{g,a}=-T^\pm_{g,-a}$ when $o(g)=2$. We also have $T^\pm_{g,a}=-T^\pm_{g,a+10}$ for all $g$, so in particular,
\begin{gather}\label{eqn:mcktht-37equiv}
	\begin{split}
	T^{\pm}_{e,7}&=-T^{\pm}_{e,3},\\
	T^{\pm}_{\hat\tau,7}&=T^{\pm}_{\hat\tau,3},\\
	T^{\pm}_{\hat\sigma,7}&=-T^{\pm}_{\hat\sigma,3}.
	\end{split}
\end{gather}

We will now verify that the series $T^X_g$ are Fourier expansions of vector-valued mock modular forms, and we will determine their shadows. For the case that $g=e$ we compute $3/40-1/12=-1/120$ and $27/40-1/12=71/120$, and see, upon comparison of (\ref{eqn:va:cnstn-Tpmeaexplicit}) with (\ref{eqn:intro:zwegers}), that $T^{\pm}_{e,1}(q)=\pm q^{-1/120}(2-\chi_0(q))$ and $T^{\pm}_{e,3}=\pm q^{71/120}\chi_1(q)$. In particular, 
\begin{gather}
	\begin{split}\label{eqn:mcktht:um-Tchi}
	2T^-_{e,1}&=2q^{-1/120}(\chi_0(q)-2),\\ 
	2T^-_{e,7}&=2q^{71/120}\chi_1(q)
	\end{split}
\end{gather}
(cf. (\ref{eqn:mcktht-37equiv})). 
Note that identities $H^X_{e,1}=2q^{-1/120}(\chi_0(q)-2)$ and $H^X_{e,7}=2q^{71/120}\chi_1(q)$ are predicted in \S5.4 of \cite{MUM}, but it is not verified there that this specification yields a mock modular form with shadow $S^X=S^X_e$. 

We will determine the modular properties of $2T^-_{e,1}$ and $2T^-_{e,7}$ by applying the results of Zwegers on $F_0$, $F_1$, $\phi_0$ and $\phi_1$ that we summarized in \S\ref{sec:mcktht:indtht}. To apply these results we first recall the expressions 
\begin{gather}
	\begin{split}\label{eqn:mcktht:um-chiFphi}
\chi_0(q) &= 2 F_0(q) - \phi_0(-q), \\
\chi_1(q) &= 2 F_1(q) + q^{-1} \phi_1(-q),
	\end{split}
\end{gather}
which are proven in \S3 of \cite{MR1577032}. (The first of these was given by Ramaujan in his last letter to Hardy, where he also mentioned the existence of a similar formula relating $\chi_1$, $F_1$ and $\phi_1$.)
Thus we obtain
\begin{gather}\label{eqn:mcktht:um-TFF}
	2T^-_{e,1}=4F_{5,1,3}(2\tau)-2F_{5,2,3}(2\tau),\\
	2T^-_{e,7}=4F_{5,1,4}(2\tau)-2F_{5,2,4}(2\tau),
\end{gather}
upon comparison of (\ref{eqn:mcktht:indtht-F5Fphi0}), (\ref{eqn:mcktht:indtht-F5Fphi1}), (\ref{eqn:mcktht:um-Tchi}) and (\ref{eqn:mcktht:um-chiFphi}). 

Applying the results of Zwegers on $F_{5,1}$ and $F_{5,2}$ recalled in \S\ref{sec:mcktht:indtht}, and the equations (\ref{eqn:mcktht:indtht-F5123shadow}) and (\ref{eqn:mcktht:indtht-F5124shadow}) in particular, we conclude that $2T^-_{e,1}$ and $2T^-_{e,7}$ are mock modular forms of weight $1/2$, with respective shadows 
given by
\begin{gather}
3(S_{30,1}+S_{30,11}+S_{30,19}+S_{30,29})(\tau),\\
3(S_{30,7}+S_{30,13}+S_{30,17}+S_{30,27})(\tau).
\end{gather}
In other words, the shadow of $T^X_{e}$ is precisely $S^X_e$, as we required to show. The modular transformation formulas for $H_{5,1}(\tau)$ and $H_{5,2}(\tau)$ given in Propositions 4.10 and 4.13 of \cite{Zwegers}, respectively, show that $T^X_e$ transforms in the desired way under $\SL_2(\ZZ)$. 

We now consider the case that $o(g)=2$. We may take $g=\hat\tau$. We again begin by using the results recalled in \S\ref{sec:mcktht:indtht} to analyze the components $T^-_{\hat\tau,1}$ and $T^-_{\hat\tau,7}$ separately. For $T^-_{\hat\tau,1}$ let
\begin{equation}
A  = \begin{pmatrix} 6 & 4 \\ 4 & 1 \end{pmatrix}, ~ a= \begin{pmatrix} 1/10 \\ 1/10 \end{pmatrix}, ~ b=\begin{pmatrix} 3/20 \\ -2/20 \end{pmatrix}, ~ c^{(1)}= \begin{pmatrix} -1  \\ 4 \end{pmatrix}, ~  c^{(2)}= \begin{pmatrix} -2 \\ 3 \end{pmatrix}.
\end{equation}
Then a direct computation using
\begin{gather}
\nu=\begin{pmatrix} k+\frac{1}{10} \\ m+\frac{1}{10} \end{pmatrix}, \;
Q(\nu)= 3 k^2+\frac{m^2}{2} + 4km+k+\frac{m}{2}+\frac3{40},\; 
B(\nu,b)= \frac{k+m}{2} + \frac{1}{10},\\
\quad \sgn \left( B(c^{(1)},\nu) \right) =\sgn \left(k+\frac{1}{10}\right),\;
\sgn  \left( B(c^{(2)},\nu) \right)= \sgn \left(-m-\frac{1}{10}\right),
\end{gather}
gives
\begin{equation}
2T^-_{\hat\tau,1}= - \frac{e(-\frac{1}{10})
}
{\eta(2 \tau)} \sum_{\nu \in a +\ZZ^2} \left( \sgn \left( B(c^{(1)},\nu) \right) - \sgn  \left( B(c^{(2)},\nu) \right) \right)
e^{2 \pi i Q(\nu) \tau+ 2 \pi i B(\nu,b)}.
\end{equation}
Comparing this to the indefinite theta function construction (\ref{eqn:mcktht:indtht-vartheta}) we find that
\begin{gather} \label{theta}
	\begin{split}
&  \vartheta_{a,b}^{c^{(1)},c^{(2)}}(\tau) = - 
e(\tfrac{1}{10})\eta(2 \tau) 2T^-_{\hat\tau,1}(\tau)  \\
 & +\sum_{ \nu \in a + \ZZ^2}
 \left(\sum_{k=1}^2(-1)^k
\sgn (B(c^{(k)},\nu))  \beta \left(  -  \frac{B(c^{(k)},\nu)^2 \Im (\tau)}{Q(c^{(k)})} \right) 
\right)
q^{Q(\nu)}e^{2 \pi i B(\nu, b)}.
 	\end{split}
 \end{gather}

We now use Proposition \ref{prop:mcktht-Zwegersprop} to rewrite the terms involving $c^{(1)}$ and $c^{(2)}$ in the second line of (\ref{theta}). For the term with $c^{(1)}$ the set $P_0$ of Proposition \ref{prop:mcktht-Zwegersprop} has one element, $\mu_0=\frac{1}{10} \left(\begin{smallmatrix} -9 \\ 1 \end{smallmatrix}\right)$, and we find $\langle c^{(1)} \rangle_\ZZ^\perp= \left\{\left(\begin{smallmatrix} 0 \\ m \end{smallmatrix}\right)\mid m \in \ZZ\right\}$, $b^\perp= \frac12\left(\begin{smallmatrix} 0 \\ {1} \end{smallmatrix}\right)$ and $\mu_0^\perp=\frac12\left( \begin{smallmatrix} 0 \\- {7} \end{smallmatrix}\right)$.
Thus
\begin{equation}
\sum_{\xi \in \mu_0^\perp + \langle c \rangle_\ZZ^\perp} e^{2 \pi i Q(\xi) \tau + 2 \pi i B(\xi,b^\perp)}
=
e(-\tfrac{1}{4})
\sum_{m \in \ZZ} (-1)^m q^{(m-1/2)^2/2}=0,
\end{equation}
so this term vanishes.

For the term with $c^{(2)}$ the set $P_0$ consists of three elements, $\mu_0=\frac{1}{10}\binom{1}{1},\frac{1}{10}\binom{1}{11},\frac{1}{10}\binom{1}{21}$, and we have
$B(c^{(2)},\mu_0)/2 Q(c^{(2)})= \frac1{30}, \frac{11}{30}, \frac{21}{20}$, in the respective cases. The last value of $\mu_0$ also leads to a vanishing contribution, while the other two
values lead to 
\begin{equation}
-
e(\tfrac{1}{12})R_{\frac{1}{30},-\frac{1}{2}}(15 \tau) \eta(2 \tau) -
e(-\tfrac{1}{12})R_{\frac{11}{30},-\frac{1}{2}}(15 \tau) \eta(2 \tau),
\end{equation}
which we see by applying Euler's identity
\begin{equation}
q^{1/12} \sum_{k \in \ZZ} (-1)^k q^{3k^2+k}= \eta(2 \tau) .
\end{equation}
We thus have
\begin{equation}\label{eqn:mcktht:um-tau1}
- 
e(-\tfrac{1}{10})\frac{\vartheta^{c^{(1)},c^{(2)}}_{a,b}(\tau)}{\eta(2 \tau)} = 2T^-_{\hat\tau,1} - 
e(-\tfrac{1}{60})R_{\frac{1}{30},-\frac{1}{2}}(15 \tau) - 
e(-\tfrac{11}{60})
R_{\frac{11}{30},-\frac{1}{2}}(15 \tau).
\end{equation}
In particular, $T^-_{\hat\tau,1}$ is the Fourier expansion of a holomorphic function on $\HH$, which we henceforth denote $T^-_{\hat\tau,1}(\tau)$. 

Since $T^-_{\hat\tau,1}(\tau)$ is holomorphic, the function (\ref{eqn:mcktht:um-tau1}) is a harmonic weak Maass form of weight $1/2$, according to Proposition 4.2 of \cite{Zwegers}. (Cf. also \S\ref{sec:mcktht:maass}.) Thus we are in a directly similar situation to that encountered at the end of \S\ref{sec:mcktht:indtht}. Namely, we have that $T^-_{\hat\tau,1}(\tau)$ is a mock modular form of weight $1/2$ (for some congruence subgroup of $\SL_2(\ZZ)$), and the second and third summands of the right hand side of (\ref{eqn:mcktht:um-tau1}) comprise the Eichler integral of its shadow. Applying (\ref{eqn:mcktht-Rabgab}), (\ref{eqn:mcktht-gab}) and (\ref{eqn:mcktht-gabSmr}), 
and also 
\begin{gather}
e(-\tfrac{1}{60})g_{\frac{1}{30},\frac{1}{2}}(15 \tau)+e(-\tfrac{11}{60})g_{\frac{11}{30},\frac{1}{2}}(15 \tau)
=\frac{1}{30}\left(
S_{30,1}+S_{30,11}+S_{30,19}+S_{30,29}
\right)(\tau),
\end{gather}
we conclude that the shadow of $2T^-_{\hat\tau,1}(\tau)$ is indeed $S^X_{\hat\tau,1}(\tau)$ (cf. (\ref{eqn:mcktht-SXg})).

For $T^-_{\hat\tau,7}$ we take $A$, $b$, $c^{(1)}$,  $c^{(2)}$ as before but set $a=\frac{1}{10}\binom{3}{3}$.
We now have 
\begin{gather}
\nu=\begin{pmatrix} k+\frac{3}{10} \\ m+\frac{3}{10} \end{pmatrix},\;
 Q(\nu)= 3 k^2 + \frac{m^2}{2}+ 4km+3k+\frac{3 m}{2}-\frac{27}{40},\;
B(\nu,b)= \frac{k+m}{2} + \frac{3}{10},\\
\sgn \left( B(c^{(1)},\nu) \right) =\sgn (k+3/10),\;\sgn  \left( B(c^{(2)},\nu) \right)= \sgn (-m-3/10).
\end{gather}

Proceeding as we did for $T^-_{\hat\tau,1}$, the contribution from the $c^{(1)}$ term vanishes again.
For the $c^{(2)}$ term we find that $P_0$ consists of the three values $\mu_0=\frac{1}{10}\binom{3}{3},\frac{1}{10}\binom{3}{13},\frac{1}{10}\binom{3}{23}$, and we have
$B(c^{(2)},\mu_0)/2 Q(c^{(2)})= \frac{3}{30}, \frac{13}{30}, \frac{23}{20}$, respectively.  The first value of $\mu_0$ leads to a vanishing contribution while the
other two terms lead to
\begin{equation}
- 
e(-\tfrac{3}{10})\frac{\vartheta^{c^{(1)},c^{(2)}}_{a,b}(\tau)}{\eta(2 \tau)} = 2T^-_{\hat\tau,7} - 
e(-\tfrac{13}{60})R_{\frac{13}{30},-\frac{1}{2}}(15 \tau) - 
e(-\tfrac{23}{60})R_{\frac{23}{30},-\frac{1}{2}}(15 \tau).
\end{equation}

We conclude thus that $T^-_{\hat\tau,7}$ is a the Fourier expansion of a mock modular form of weight $1/2$, 
and using
\begin{gather}
e(-\tfrac{13}{60})g_{\frac{13}{30},\frac{1}{2}}(15 \tau)+e(-\tfrac{23}{60})g_{\frac{23}{30},\frac{1}{2}}(15 \tau)
=\frac{1}{30}\left(
S_{30,7}+S_{30,13}+S_{30,17}+S_{30,23}
\right)(\tau)
\end{gather}
we see that the shadow of $2T^-_{\hat\tau,1}(\tau)$ is $S^X_{\hat\tau,1}(\tau)$ (cf. (\ref{eqn:mcktht-SXg})). So we have verified that the shadow of $2T^-_{g}=(2T^-_{g,r})$ is $S^X_g=(S^X_{g,r})$ for $o(g)=2$. 

Corollary 2.9 of \cite{Zwegers} details the modular transformation properties of the indefinite theta functions $\vartheta^{c^{(1)},c^{(2)}}_{a,b}(\tau)$. Applying these formulas, much as in the proofs of Propositions 4.10 and 4.13. in \cite{Zwegers}, we see that $2T^-_{\hat\tau}$ transforms in the desired way under the action of $\Gamma_0(2)$. 

Corollary 2.9 also enables us to compute the expansion of $2T^-_{\hat\tau}$ at the cusp of $\Gamma_0(2)$ represented by $0$. We ultimately find that both $T^-_{\hat\tau,1}(\tau)$ and $T^-_{\hat\tau,7}(\tau)$ vanish as $\tau\to 0$. Thus $2T^-_{\hat\tau}$ has no poles away from the infinite cusp.

It remains to consider the case $o(g)=3$, but this can be handled by applying classical results on positive-definite theta functions, since the formula (\ref{eqn:va:cnstn-Tpmsigmaaexplicit}) gives $T^-_{\hat\sigma,1}$ and $T^-_{\hat\sigma,7}$ explicitly in terms of the Dedekind eta function and the theta series of a rank one lattice. We easily check that these functions transform in the desired way under $\Gamma_0(3)$, and have no poles away from the infinite cusp of $\Gamma_0(3)$. In particular, $2T^-_{\hat\sigma}$ is modular, and has vanishing shadow. 
\end{proof}

We are now ready to prove our main results.

\begin{proof}[Proof of Theorem \ref{thm:intro-maintheorem}]
Proposition \ref{prop:mcktht:um-TXg} demonstrates that the functions $2T^X_g$ are mock modular forms of weight $1/2$ with the claimed shadows, multiplier systems, and polar parts. It remains to verify that they are the unique such functions.

The uniqueness in case $g=e$ is shown in Corollary 4.2 of \cite{MUM}, 
using the fact (see Theorem 9.7 in \cite{Dabholkar:2012nd}) that there are no weak Jacobi forms of weight 1. We will give a different (but certainly related) argument here.

Consider first the case that $o(g)$ is $1$ or $2$. It suffices to show that if $h=(h_r)$ is a modular form of weight $1/2$, transforming with the same multiplier system as $H^X$ under $\Gamma_0(2)$, with $h_r$ vanishing whenever $r$ does not belong to  
\begin{gather}\label{eqn:mcktht:um-rrestriction}
\{\pm 1,\pm 7,\pm 11,\pm 13,\pm 17,\pm 19,\pm 23, \pm 29 \},
\end{gather}
then $h$ vanishes identically. The multiplier system for $H^X$ is trivial when restricted to $\Gamma(120)$, so the components $h_r$ are modular forms for $\Gamma_0(2)\cap\Gamma(120)=\Gamma(120)$. Satz 5.2. of \cite{Sko_Thesis} is an effective version of the celebrated theorem of Serre--Stark \cite{MR0472707} on modular forms of weight $1/2$ for congruence subgroups of $\SL_2(\ZZ)$. It tells us that the space of modular forms of weight $1/2$ for $\Gamma(120)$ is spanned by certain linear combinations of the {\em thetanullwerte} $\theta^0_{n,r}(\tau):=\theta_{n,r}(\tau,0)$, and the only $n$ that can appear are those that divide $30$. On the other hand, the restriction (\ref{eqn:mcktht:um-rrestriction}) implies that any non-zero component $h_r$ must belong to one of $q^{-1/120}\CC[[q]]$ or $q^{71/120}\CC[[q]]$. We conclude that all the $h_r$ are necessarily zero by checking, using 
\begin{gather}\label{eqn:mcktht:um-thetanullwerte}
\theta_{n,r}^0(\tau)=\sum_{k\in\ZZ}q^{(2kn+r)^2/4n},
\end{gather}
that none of the $\theta^0_{n,r}$ belong to either space, for $n$ a divisor of $30$.

The case that $o(g)=3$ is very similar, except that the $h_r$ are now modular forms on $\Gamma_0(9)\cap \Gamma(120)$, which contains $\Gamma(360)$, and the relevant thetanullwerte are those $\theta_{n,r}^0$ with $n$ a divisor of $90$. We easily check using (\ref{eqn:mcktht:um-thetanullwerte}) that there are non-zero possibilities for $h_r$, and this completes the proof.
\end{proof}

\begin{proof}[Proof of Theorem \ref{thm:intro-rammcktht}]
Taking now (\ref{eqn:intro-HXg}) as the definition of $H^X_g$, the identities (\ref{eqn:intro-rammcktht1}) follow directly from the definition (\ref{eqn:mcktht-TXg}) of $T^X_g$, and the explicit expressions (\ref{eqn:va:cnstn-Tpmeaexplicit}) for the components of $T^X_e$. 

The identities (\ref{eqn:intro-rammcktht2}) follow from the characterization of $H^X_g$ for $o(g)=2$ that is entailed in Theorem \ref{thm:intro-maintheorem}. Indeed, using Zwegers' results (viz., Propositions 4.10 and 4.13 in \cite{Zwegers}) on the modularity of $\phi_0(-q)$ and $\phi_1(-q)$, we see that the function defined by the right hand side of (\ref{eqn:intro-rammcktht2}) is a vector-valued mock modular form with exactly the same shadow as $2T^X_{\hat\tau}$, transforming with the same multiplier system under $\Gamma_0(2)$, and having the same polar parts at both the infinite and non-infinite cusps of $\Gamma_0(2)$. So it must coincide with $H^X_{2A,1}=2T^X_{\hat\tau}$ according to Theorem \ref{thm:intro-maintheorem}. This completes the proof.
\end{proof}

\begin{proof}[Proof of Corollary \ref{cor:intro-qseriesid}]
Andrews established Hecke-type ``double sum'' identities for $\phi_0$ and $\phi_1$ in \cite{MR814916}. Rewriting these slightly, we find
\begin{gather}
	\phi_0(-q)=\frac{(q;q)_{\infty}}{(q^2;q^2)_\infty^2}\label{eqn:mcktht:um-phi0Hecke}
	\left( \sum_{k,m \ge 0} -
\sum_{k,m <0} \right)_{\text{$k=m$ mod $2$}}
(-1)^m q^{k^2/2+m^2/2+4km+k/2+3m/2},\\
	-q^{-1}\phi_1(-q)=\frac{(q;q)_{\infty}}{(q^2;q^2)_\infty^2}\label{eqn:mcktht:um-phi1Hecke}
	\left( \sum_{k,m \ge 0} -
\sum_{k,m <0} \right)_{\text{$k=m$ mod $2$}}
(-1)^m q^{k^2/2+m^2/2+4km+3k/2+5m/2}.
\end{gather}
Armed with the identities (\ref{eqn:intro-rammcktht2}), we obtain (\ref{eqn:intro-qseriesid1}) and (\ref{eqn:intro-qseriesid7}) by comparing (\ref{eqn:mcktht:um-phi0Hecke}) and (\ref{eqn:mcktht:um-phi1Hecke}) with the explicit expression (\ref{eqn:va:cnstn-Tpmtauaexplicit}) for the components of $T^X_{\hat\tau}$.
\end{proof}

\section*{Acknowledgement}
We thank Miranda Cheng for particularly helpful discussions and advice that took place in the early stages of this work. We also thank Ching Hung Lam for discussions on the vertex operator algebra structure here employed. The research of J.D. was supported in part by the Simons Foundation (\#316779).
Both authors gratefully acknowledge support from the U.S. National Science Foundation (grants 1203162 and 1214409).

\clearpage

\appendix

\begin{table}[h]
\section{Coefficients}\label{sec:coeffs}
\vspace{-12pt}
\begin{minipage}[t]{0.49\linewidth}
\centering
\caption{\label{tab:coeffs:30+6,10,15_1}$H^{X}_{g,1}$, $X=E_8^3$}\smallskip
\begin{tabular}{r|rrr}
\toprule
$[g]$	&1A	&2A	&3A	\\
\midrule
$\G_g$	&$1|1$	&$2|1$	&$3|3$	\\
\midrule
-1	&-2	&-2	&-2	\\
119	&2	&2	&2	\\
239	&2	&-2	&2	\\
359	&4	&0	&-2	\\
479	&2	&-2	&2	\\
599	&6	&2	&0	\\
719	&4	&0	&-2	\\
839	&6	&2	&0	\\
959	&6	&-2	&0	\\
1079	&10	&2	&-2	\\
1199	&6	&-2	&0	\\
1319	&12	&0	&0	\\
1439	&10	&-2	&-2	\\
1559	&14	&2	&2	\\
1679	&14	&-2	&2	\\
1799	&18	&2	&0	\\
1919	&14	&-2	&2	\\
2039	&24	&4	&0	\\
2159	&22	&-2	&-2	\\
2279	&26	&2	&2	\\
2399	&26	&-2	&2	\\
2519	&34	&2	&-2	\\
2639	&30	&-2	&0	\\
2759	&42	&2	&0	\\
2879	&40	&-4	&-2	\\
2999	&48	&4	&0	\\
3119	&48	&-4	&0	\\
3239	&58	&2	&-2	\\
3359	&56	&-4	&2	\\
3479	&72	&4	&0	\\
3599	&70	&-2	&-2	\\
3719	&80	&4	&2	\\
3839	&84	&-4	&0	\\
3959	&100	&4	&-2	\\
4079	&96	&-4	&0	\\
4199	&116	&4	&2	\\
4319	&116	&-4	&-4	\\
4439	&134	&6	&2	\\
4559	&140	&-4	&2	\\
\bottomrule
\end{tabular}
\end{minipage}
\begin{minipage}[t]{0.49\linewidth}
\centering
\caption{\label{tab:coeffs:30+6,10,15_7}$H^{X}_{g,7}$, $X=E_8^3$}\smallskip
\begin{tabular}{r|rrr}
\toprule
$[g]$	&1A	&2A	&3A	\\
\midrule
$\G_g$	&$1|1$	&$2|1$	&$3|3$	\\
\midrule
71	&2	&-2	&2	\\
191	&4	&0	&-2	\\
311	&4	&0	&-2	\\
431	&6	&2	&0	\\
551	&6	&-2	&0	\\
671	&8	&0	&2	\\
791	&8	&0	&2	\\
911	&12	&0	&0	\\
1031	&10	&-2	&-2	\\
1151	&14	&2	&2	\\
1271	&16	&0	&-2	\\
1391	&18	&2	&0	\\
1511	&18	&-2	&0	\\
1631	&24	&0	&0	\\
1751	&24	&0	&0	\\
1871	&30	&2	&0	\\
1991	&30	&-2	&0	\\
2111	&36	&0	&0	\\
2231	&38	&-2	&2	\\
2351	&46	&2	&-2	\\
2471	&46	&-2	&-2	\\
2591	&54	&2	&0	\\
2711	&60	&0	&0	\\
2831	&66	&2	&0	\\
2951	&68	&-4	&2	\\
3071	&82	&2	&-2	\\
3191	&84	&0	&0	\\
3311	&98	&2	&2	\\
3431	&102	&-2	&0	\\
3551	&114	&2	&0	\\
3671	&122	&-2	&2	\\
3791	&138	&2	&0	\\
3911	&144	&-4	&0	\\
4031	&162	&2	&0	\\
4151	&174	&-2	&0	\\
4271	&192	&4	&0	\\
4391	&200	&-4	&2	\\
4511	&226	&2	&-2	\\
4631	&238	&-2	&-2	\\
\bottomrule
\end{tabular}
\end{minipage}
\end{table}

\clearpage



\providecommand{\href}[2]{#2}\begingroup\raggedright\endgroup




\end{document}